\documentclass[3p]{elsarticle}
\usepackage{epsfig}
\usepackage{amsfonts}
\usepackage{amssymb}
\usepackage{graphicx}
\usepackage{algorithmic}
\usepackage{algorithm}
\usepackage{url}
\usepackage{epsfig}
\usepackage{epsfig}
\usepackage{booktabs}
\usepackage{amssymb,amsmath,amsthm}
\usepackage{latexsym}
\usepackage{fancyhdr}
\usepackage{fancybox}
\usepackage{graphics}
\usepackage{palatino}
\usepackage[dvips]{color}
\usepackage{float}
\usepackage{enumerate}
\usepackage{color}
\usepackage{epsfig}
\usepackage{multirow}
\usepackage{url}
\usepackage{rotating}
\usepackage{pdflscape}
\usepackage{longtable}
\usepackage[english]{babel,minitoc}

\newtheorem{proposition}{Proposition}
\newtheorem{theorem}{Theorem}
\newtheorem{corollary}{Corollary}

\newtheorem{remark}{Remark}
\newtheorem{assumption}{Assumption}

\journal{European Journal of Operational Research}
\input{tcilatex}
\begin{document}

\begin{frontmatter}

\title{DC approximation approaches for sparse optimization }
\author[ltha]{LE THI Hoai An}
\ead{hoai-an.le-thi@univ-lorraine.fr}
\author[pdt]{PHAM DINH Tao}
\ead{pham@insa-rouen.fr}
\author[ltha] {LE Hoai Minh}
\ead{minh.le@univ-lorraine.fr}
\author[ltha] {VO Xuan Thanh}
\ead{xuan-thanh.vo@univ-lorraine.fr}

\address[ltha]{Laboratory of Theoretical and Applied Computer Science EA 3097\\
  University of Lorraine, Ile du Saulcy, 57045 Metz, France\\}

\address[pdt]{Laboratory of Mathematics, INSA - Rouen,
University of Normandie \\
  76801 Saint-Etienne-du-Rouvray Cedex, France\\}

\begin{abstract}
Sparse optimization refers to an optimization problem involving the
zero-norm in objective or constraints. In this paper, nonconvex approximation approaches for sparse optimization have been studied
with a unifying point of view in DC (Difference of Convex functions) programming framework. Considering a common DC approximation of
the zero-norm including all standard sparse inducing penalty functions, we studied the consistency between global minimums (resp. local
minimums) of approximate and original problems. We showed that, in several cases, some global minimizers (resp. local
minimizers) of the approximate problem are also those of the original problem. Using  exact penalty techniques in DC
programming, we proved stronger results for some particular approximations, namely, the approximate problem, with suitable parameters,
is equivalent to the original problem. The efficiency of several sparse inducing penalty functions have been fully analyzed. Four DCA (DC Algorithm)
schemes were developed that cover all standard algorithms in nonconvex sparse approximation approaches as special versions. They can be viewed
as, an $\ell _{1}$-perturbed algorithm / reweighted-$\ell _{1}$ algorithm
 /  reweighted-$\ell _{1}$ algorithm.  We offer a unifying nonconvex approximation approach, with solid theoretical tools as well as efficient algorithms based on DC programming and DCA, to tackle the zero-norm and sparse optimization. As an application, we implemented our methods for the feature selection in SVM (Support Vector Machine) problem and performed empirical comparative numerical experiments on the proposed algorithms with various approximation functions.
\end{abstract}


\end{frontmatter}

\section{Introduction}

The zero-norm on $\mathbb{R}^{n}$, denoted $\ell _{0}$-norm or $\Vert .\Vert
_{0}$, is defined by 
\begin{equation*}
\Vert x\Vert _{0}:=\left\vert \{i=1,...,n:x_{i}\neq 0\}\right\vert ,
\end{equation*}%
where $\left\vert S\right\vert $ is the cardinality of the set $S.$ 
The $\ell _{0}$-norm is an important concept for modelling the sparsity of
data and plays a crucial role in optimization problems where one has to
select representative variables. Sparse optimization, which refers to an
optimization problem involving the $\ell _{0}$-norm in objective or
constraints, has many applications in various domains (in particular in
machine learning, image processing and finance), and draws increased
attention from many researchers in recent years. The function $\ell _{0}$,
apparently very simple, is lower-semicontinuous on $\mathbb{R}^{n},$ but its
discontinuity at the origin makes nonconvex programs involving $\Vert .\Vert
_{0}$ challenging. Note that although one uses the term \textquotedblright
norm\textquotedblright\ to design $\left\Vert .\right\Vert _{0}$, $%
\left\Vert .\right\Vert _{0}$ is not a norm in the mathematical sense.
Indeed, for all $x\in \mathbb{R}^{n}$ and $\lambda \neq 0$, one has $%
\left\Vert \lambda x\right\Vert _{0}=\left\Vert x\right\Vert _{0},$ which is
not true for a norm. \newline

Formally, a sparse optimization problem takes the form 
\begin{equation}
\inf \left\{ f(x,y)+\lambda \left\Vert x\right\Vert _{0}:(x,y)\in K\subset 
\mathbb{R}^{n}\mathbb{\times R}^{m}\ \right\} ,  \label{mainpbl0}
\end{equation}%
\cite{Weston03}\cite{Zhang06}where the function $f$ corresponds to a given
criterion and $\lambda $ is a positive number, called the regularization
parameter, that makes the trade-off between the criterion $f$ and the
sparsity of $x$. In some applications, one wants to control the sparsity of
solutions, the $\ell _{0}$-term is thus put in constraint\textbf{s}, and the
corresponding optimization problem is 
\begin{equation}
\inf \{f(x,y):(x,y)\in K,\Vert x\Vert _{0}\leq k\}.  \label{l0-constraint}
\end{equation}

Let us mention some important applications of sparse optimization
corresponding to these models.

\medskip

\noindent \textit{Feature selection in classification learning: } Feature
selection is one of fundamental problems in machine learning. In many
applications such as text classification, web mining, gene expression,
micro-array analysis, combinatorial chemistry, image analysis, etc, data
sets contain a large number of features, many of which are irrelevant or
redundant. Feature selection is often applied to high-dimensional data prior
to classification learning. The main goal is to select a subset of features
of a given data set while preserving or improving the discriminative ability
of the classifier. Given a training data $\left\{ a_{i},b_{i}\right\}
_{i=1,...,q}$ where each $a_{i}\in \mathbb{R}^{n}$ is labeled by its class $%
b_{i}\in Y$, the discrete set of labels. The aim of classification\ learning
is to construct a classifier function that discriminates the data points $%
A:=\left\{ a_{i}\right\} _{i=1,...,q}$ with respect to their classes$\left\{
b_{i}\right\} _{i=1,...,q}$. The embedded feature selection in
classification consists of determining the classifier which uses as few
features as possible, that leads to a sparse optimization problem like (\ref%
{mainpbl0}).

\medskip

\noindent \textit{Sparse Regression: } Given a training data set $\left\{
b_{i},a_{i}\right\} _{i=1}^{q}$ of $q$ independent and identically
distributed samples composed of explanatory variables $a_{i}\in \mathbb{R}%
^{n}$ (inputs) and response variables $b_{i}\in \mathbb{R}$ (ouputs). Let $%
b:=(b_{i})_{i=1,...,q}$ denote the vector of outputs and $%
A:=(a_{i,j})_{i=1,...,q}^{j=1,...,n}$ denote the matrix of inputs. The
problem of the regression consists in looking for a relation which can
possibly exist between $A$ and $b$, in other words, relating $b$ to a
function of $A$ and a model parameter $x$.\ Such a model parameter $x$ can
be obtained by solving the optimization problem 
\begin{equation}
\min \left\{ f(x):=\sum_{i=1}^{q}L(b_{i},a_{i}^{T}x):x\in \mathbb{R}%
^{n}\right\},  \label{rg}
\end{equation}%
where $L:\mathbb{R}^{n}\rightarrow \mathbb{R}$ is called loss function. The 
\textit{sparse regression} problem aims to find a sparse solution of the
above regression model, it takes the form of (\ref{mainpbl0}): 
\begin{equation}
\min_{x\in \mathbb{R}^{n}}\left\{ \sum_{i=1}^{q}L(b_{i},a_{i}^{T}x)+\rho
\left\Vert x\right\Vert _{0}\right\} .  \label{LS}
\end{equation}


\noindent \textit{Sparse Fisher Linear Discriminant Analysis: } Discriminant
analysis captures the relationship between multiple independent variables
and a categorical dependent variable in the usual multivariate way, by
forming a composite of the independent variables. Given a set of $q$
independent and identically distributed samples composed of explanatory
variables $a_{i}\in \mathbb{R}^{n}$ and binary response variables $b_{i}\in
\left\{ -1,1\right\} $. The idea of Fisher linear discriminant analysis is
to determine a projection of variables onto a straight line that best
separables the two classes. The line is so determined to maximize the ratio
of the variances of between and within classes in this projection, i.e.
maximize the function $f(\alpha )=\frac{\langle \alpha ,S_{B}\alpha \rangle 
}{\langle \alpha ,S_{W}\alpha \rangle },$ where $S_{B}$ and $S_{W}$ are,
respectively, the between and within classes scatter matrix (they are
symmetric positive semidefinite) given by 
\begin{equation*}
S_{B}:=(q_{+}-q_{-})(q_{+}-q_{-})^{T} , \: S_{W}=S_{+}+S_{-},
\end{equation*}%
\begin{equation*}
S_{+}=\sum_{i=1,b_{i}=+1}^{q}(x_{i}-q_{+})(x_{i}-q_{+})^{T}, \:
S_{-}=\sum_{i=1,b_{i}=-1}^{q}(x_{i}-q_{-})(x_{i}-q_{-})^{T}.
\end{equation*}%
Here, for $j\in \left\{ \pm \right\} $, $q_{j}$ is the mean vector of class $%
j$, $l_{j}$ is the number of labeled samples in class $j$. If $\alpha $ is
an optimal solution of the problem, then the classifier is given by $%
F(a)=\alpha ^{T}a+c$, $c=0.5\alpha ^{T}(q_{+}-q_{-})$.\newline
The sparse Fisher Discriminant model is defined by ($\rho >0$ ) 
\begin{equation*}
\min \{ \alpha ^{T}S_{W}\alpha +\rho \left\Vert \alpha \right\Vert _{0} :
\alpha ^{T}(q_{+}-q_{-})=b\}.
\end{equation*}


\medskip

\noindent \textit{Compressed sensing: } Compressed sensing refers to
techniques for efficiently acquiring and reconstructing signals via the
resolution of underdetermined linear systems. Compressed\textbf{\ }sensing
concerns sparse signal representation, sparse signal recovery and sparse
dictionary learning which can be formulated as sparse optimization problems
of the form (\ref{mainpbl0}).

\medskip

\noindent \textit{Portfolio selection problem with cardinality constraint: }
In portfolio selection problem, given a set of available securities or
assets, we want to find the optimum way of investing a particular amount of
money in these assets. Each of the different ways to diversify this money
among the several assets is called a portfolio. In portfolio management one
wants to limit the number of assets to be investigated in the portfolio,
that leads to a problem of the form (\ref{l0-constraint}).

\medskip

\noindent \textit{Other applications: } Other applications of sparse
optimization include Sensor networks (\cite{Bajwa06,Baron06}), Error
correction (\cite{Candes05,Candes06}), Digital photography (\cite{Takhar06}%
), etc.

\medskip

\textbf{Existing works}. During the last two decades, research is very
active in models and methods optimization involving the zero-norm. Works can
be divided into three categories according to the way to treat the
zero-norm: convex approximation, nonconvex approximation, and nonconvex
exact reformulation.

In the machine learning community, one of the best known approaches,
belonging to the group "convex approximation", is the $\ell _{1}$
regularization approach proposed in \cite{TIB96} in the context of linear
regression, called LASSO (Least Absolute Shrinkage and Selection Operator),
which consists in replacing the $\ell _{0}$ term $\left\Vert x\right\Vert
_{0}$ by $\left\Vert x\right\Vert _{1}$, the $\ell _{1}$ -norm of the vector 
$x$. In \cite{Gribonval03}, the authors have proved that, under suitable
assumptions, a solution of the $\ell _{0}$- regularizer problem over a
polyhedral set can be obtained by solving the $\ell _{1}$- regularizer
problem. However, these assumptions are quite restrictive. Since its
introduction, several works have been developed to study the $\ell _{1}$%
-regularization technique, from the theoretical point of view to efficient
computational methods (see \cite{HAS09}, Chapter 18 for more discussions on $%
\ell _{1}$-regularized methods). The LASSO penalty has been shown to be, in
certain cases, inconsistent for variable selection and biased \cite{ZOU06}.
Hence, the Adaptive LASSO is introduced in \cite{ZOU06} in which adaptive
weights are used for penalizing different coefficients in the $\ell _{1}$%
-penalty. 

At the same time, nonconvex continuous approaches, belonging to the second
group "nonconvex approximation" (the $\ell _{0}$ term $\left\Vert
x\right\Vert _{0}$ is approximated by a nonconvex continuous function) were
extensively developed. A variety of sparsity-inducing penalty functions have
been proposed to approximate the $\ell _{0}$ term: exponential concave
function \cite{Bradley-Mangasarian}, $\ell _{p}$-norm with $0<p<1$ \cite%
{Fu98} and $p<0$ \cite{Rao99}, Smoothly Clipped Absolute Deviation (SCAD) 
\cite{FAN01}, Logarithmic function \cite{Weston03}, Capped-$\ell _{1}$ \cite%
{Peleg08} (see (\ref{log-lp}), (\ref{scad}) and Table~\ref{tab:app-form} in
Section~\ref{approxim} for the definition of these functions). Using these
approximations, several algorithms have been developed for resulting
optimization problems, most of them are in the context of feature selection
in classification, sparse regressions or more especially for sparse signal
recovery: Successive Linear Approximation (SLA) algorithm \cite%
{Bradley-Mangasarian}, DCA (Difference of Convex functions Algorithm) based
algorithms \cite%
{XIA10,COLO06,GAS09,GUAN-GRAY13,Leetal2013,LeThietal2008a,LeThietal2009,LTNL13,LTN13,Neumann05,Ong12}%
, Local Linear Approximation (LLA) \cite{Zou08}, Two-stage $\ell _{1}$ \cite%
{Zhang09}, Adaptive Lasso \cite{ZOU06}, reweighted-$\ell _{1}$ algorithms 
\cite{Candes08}), reweighted- $\ell _{2}$ algorithms such as Focal
Underdetermined System Solver (FOCUSS) (\cite{Rao97,Rao99,Rao03}),
Iteratively reweighted least squares (IRLS) and Local Quadratic
Approximation (LQA) algorithm \cite{FAN01,Zou08}.

In the third category named nonconvex exact reformulation approaches, the $%
\ell _{0}$-regularized problem is reformulated as a continuous nonconvex
program. There are a few works in this category. In \cite{Mangasarian96},
the author reformulated the problem (\ref{mainpbl0}) in the context of
feature selection in SVM as a linear program with equilibrium constraints
(LPEC). However, this reformulation is generally intractable for large-scale
datasets. 
In \cite{Thiao-et-al08,PLT13} an exact penalty technique in DC programming
is used to reformulate (\ref{mainpbl0}) and (\ref{l0-constraint}) as DC
programs. In \cite{Thiao-et-al11} this technique is used for Sparse
Eigenvalue problem with $\ell _{0}$-norm in constraint functions 
\begin{equation}
\max \{x^{T}Ax:x^{T}x=1,\left\Vert x\right\Vert _{0}\leq k\},
\label{sparse eigenvalue program}
\end{equation}%
where $A\in $ $\mathbb{R}^{n\times n}$ is symmetric and $k$ an integer, and
a DCA based algorithm was investigated for the resulting problem.

Beside the three above categories, heuristic methods are developed to tackle
directly the original problem (\ref{mainpbl0}) by greedy based algorithms,
e.g. matching pursuit, orthogonal matching pursuit, \cite%
{Mallat-Zang93,Pati93}, etc.

Convex regularization approaches involve convex optimization problems which
are so far "easy" to solve, but they don't attain the solution of the $\ell
_{0}$-regularizer problem. Nonconvex approximations are, in general, deeper
than convex relaxations, and then can produce good sparsity, but the
resulting optimization problems are still difficult since they are nonconvex
and there are many local minima which are not global. Many issues have not
yet been studied or proved in the existing approximation approaches. First,
the consistency between the approximate problems and the original problem is
a very important question but still is open. Only a weak result has been
proved for two special cases in \cite{Bradley98b} (resp. \cite{Rinaldi10})
when $f$ is concave, bounded below on a polyhedral convex set $K$ and the
approximation term is an exponential concave function 
(resp. a logarithm function and/or $\ell _{p}$-norm ($p<1$)). It has been
shown in these works that the intersection of the solution sets of the
approximate problem and the original problem is nonempty. Moreover no result
on the consistency between local minimum of approximate and original
problems has been available, while most of the proposed algorithms furnish
local minima. Second, several existing algorithms lack a rigorous mathematical proof of
convergence. Hence the choice of a \textquotedblright
good\textquotedblright\ approximation remains relevant. Two crucial
questions should be studied for solving large scale problems, that are, 
\textit{how to suitably approximate the zero-norm} and \textit{which
computational method to use} for solving the resulting optimization problem.
The development of new models and algorithms for sparse optimization
problems is always a challenge for researchers in optimization and machine
learning.

\textbf{Our contributions. } We consider in this paper the problem (\ref%
{mainpbl0}) where $K$ is a polyhedral convex set in $\mathbb{R}^{n}\times 
\mathbb{R}^{m}$ and $f$ is a finite DC function on $\mathbb{R}^{n}\times 
\mathbb{R}^{m}$. We address all issues cited above for approximation
approaches and develop an unifying approach based on DC programming and DCA,
a robust, fast and scalable approach for nonconvex and nonsmooth continuous
optimization (\cite{LTP05, PLT98}). The contributions of this paper are
multiple, from both a theoretical and a computational point of view.

Firstly, considering a common DC approximate function, we prove the
consistency between the approximate problem and the original problem by
showing the link between their global minimizers as well as their local
minimizers. We demonstrate that any optimal solution of the approximate
problem is in a $\epsilon -$neighbourhood of an optimal solution to the
original problem (\ref{mainpbl0}). More strongly, if $f$ is concave and the
objective function of the approximate problem is bounded below on $K$, 
then some optimal solutions of the approximate problem are exactly solutions
of the original problem. These new results are important and very useful for
justifying the performance of approximation approaches.\ 

Secondly, we provide an in-depth analysis of usual sparsity-inducing
functions and compare them according to suitable parameter values. This
study suggests the choice of good approximations of the zero-norm as well as
that of good parameters for each approximation. A reasonable comparison via
suitable parameters identifies Capped -$\ell _{1}$ and SCAD as the best
approximations.

Thirdly, we prove, via an exact reformulation approach by exact penalty
techniques that, with suitable parameters ($\theta >\theta _{0}$), nonconvex
approximate problems resulting from Capped -$\ell _{1}$ or SCAD functions
are equivalent to the original problem. Moreover, when the set $K$ is a box,
we can show directly (without using exact penalty techniques) the
equivalence between the original problem and the approximate Capped -$\ell
_{1}$ problem and give the value of $\theta _{0}$ such that this equivalence
holds for all $\theta >\theta _{0}$. These interesting and significant
results justify our analysis on usual sparsity-inducing functions and the
pertinence of these approximation approaches. It opens the door to study
other approximation approaches which are consistent with the original
problem.

Fourthly, we develop solution methods for all DC approximation approaches.
Our algorithms are based on DC programming and DCA, because our main
motivation is to exploit the efficiency of DCA to solve this hard problem.
We propose three DCA schemes for three different formulations of a common
model to all concave approximation functions. We show that these DCA schemes
include all standard algorithms as special versions. The fourth DCA scheme
is concerned with the resulting DC program given by the DC approximation
(nonconcave piecewise linear) function in \textbf{(}\cite{LT12}). Using DC
programming framework, we unify all solution methods into DCA\textbf{, }and
then convergence properties of our algorithms are guaranteed, thanks to
general convergence results of the generic DCA scheme. It permits to
exploit, in an elegant way, the nice effect of DC decompositions of the
objective functions to design various versions of DCA. It is worth
mentioning here the flexibility/versatility of DC\ programming and DCA: the
four algorithms can be viewed as an $\ell _{1}$-perturbed algorithm / a
reweighted-$\ell _{1}$ algorithm (intimately related to the $\ell _{1}$%
-penalized LASSO approach) / a reweighted-$\ell _{2}$ algorithm in case of
convex objective functions.

Finally, as an application, we consider the problem of feature selection in
SVM and perform a careful empirical comparison of all approaches.

The rest of the paper is organized as follows. Since DC programming and DCA
is the core of our approaches, we give in Section~\ref{dca} a brief
introduction of these theoretical and algorithmic tools. The consistency
between approximate problems and the original one, the link between their
global minimizer as well as their local minimizer are studied in Section~\ref%
{approxim}, while a comparative analysis on usual approximations is
discussed in Section~\ref{Sect.compare}. A deeper study on Capped-$\ell _{1}$
approximation and {the relation between some approximate problems and exact
penalty approaches is presented in} Section~{\ref{exact}. Solution methods
based on DCA are developed in Section \ref{DCA-prox}, while the application
of the proposed algorithms for feature selection in SVM and numerical
experiments are described in Section~\ref{num}. At last, }Section{\textbf{~}%
\ref{conclu} concludes the paper. 
}

\section{Outline of DC programming and DCA}

\label{dca}

Let $X$\textbf{\ }be the Euclidean space\textbf{\ }$\mathrm{I\!R}^{n}$%
\textbf{\ }equipped with the canonical inner product\textbf{\ }$\langle
.,.\rangle $\textbf{\ }and its Euclidean norm\textbf{\ }$\left\Vert
.\right\Vert .$ The dual space of $X$, denoted by\textbf{\ }$Y$\textbf{, }%
can be identified with\textbf{\ }$X$\textbf{\ }itself\textbf{.}

DC(Difference of Convex functions) Programming and DCA (DC Algorithms),
which constitute the backbone of nonconvex programming and global
optimization, are introduced in 1985 by Pham Dinh Tao in the preliminary
state, and extensively developed by Le Thi Hoai An and Pham Dinh Tao since
1994 (\cite{Block2013,lethi-website,EJOR2002,COCONUT,SIOPT2,LTP05,Concave
costEJOR07,DiscreteMath08,ConvergenceDCA,Lethietal2012b,Lethietal2012c,LTP2013,LeThietal2014,Acta,PLT98}
and references quoted therein). Their original key idea relies on the
structure DC of objective function and constraint functions in nonconvex
programs which are explored and exploited in a deep and suitable way. The
resulting DCA introduces the nice and elegant concept of approximating a
nonconvex (DC) program by a sequence of convex ones: each iteration of DCA
requires solution of a convex program.

Their popularity resides in their rich, deep and rigorous mathematical
foundations, and the versatility/flexibility, robustness, and efficiency of
DCA's compared to existing methods, their adaptation to specific structures
of addressed problems and their ability to solve real-world large-scale
nonconvex programs. Recent developments in convex programming are mainly
devoted to reformulation techniques and scalable algorithms in order to
handle large-scale problems. Obviously, they allow for enhancement of DC
programming and DCA in high dimensional nonconvex programming.

Standard DC\ programs are of the form: 
\begin{equation*}
\alpha =\inf \{f(x):=g(x)-h(x):x\in \mathrm{I\!R}^{n}\}\qquad \qquad (P_{dc})
\end{equation*}%
where $g,h\in \Gamma _{0}$($\mathrm{I\!R}^{n})$, the convex cone of all
lower semicontinuous proper (i.e., not identically equal to $+\infty )$
convex functions defined on $\mathrm{I\!R}^{n}$ and taking values in $%
\mathrm{I\!R\cup \{+\infty \}.}$ Such a function $f$ is called a DC
function, and $g-h$ a DC~decomposition of $f$ while $g$ and $h$ are the DC
components of $f.$ The convex constraint $x\in C$ can be incorporated in the
objective function of $(P_{dc})$ by using the indicator function of $C$
denoted by $\chi _{C}$ which is defined by $\chi _{C}(x)=0$ if $x\in C$, and 
$+\infty $ otherwise~: 
\begin{equation*}
\inf \{f(x):=g(x)-h(x):x\in C\mathrm{\ }\}=\inf \{\chi
_{C}(x)+g(x)-h(x):x\in \mathrm{I\!R}^{n}\}.
\end{equation*}

The vector space of DC functions, $DC({\mathbb{R}}^{n})=\Gamma _{0}({\mathbb{%
R}}^{n})-\Gamma _{0}({\mathbb{R}}^{n})$, forms a wide class encompassing
most real-life objective functions and is closed with respect to usual
operations in optimization. DC programming constitutes so an extension of
convex programming, sufficiently large to cover most nonconvex programs (%
\cite{lethi-website,lethithesis,PLT97,COCONUT,SIOPT2,LTP05,Acta,PLT98} and
references quoted therein), but not too in order to leverage the powerful
arsenal of the latter.

The conjugate of $\varphi $, denoted by $\varphi ^{\ast },$is given by%
\begin{equation*}
\varphi ^{\ast }(y):=\sup \{\langle x,y\rangle -\varphi (x):x\in \mathrm{I\!R%
}^{n}\},\forall y\in \mathrm{I\!R}^{n}.
\end{equation*}

DC duality associates the primal DC program $(P_{dc})$ with its dual $%
(D_{dc})$, which is also a DC program with the same optimal value and
defined by 
\begin{equation*}
\alpha =\inf \{h^{\ast }(y)-g^{\ast }(y):y\in \mathrm{I\!R}^{n}\},\qquad
\qquad (D_{dc})
\end{equation*}%
and studies the relation between primal and dual solution sets denoted by $
{\cal P}$ and ${\cal D}$ respectively. In DC programming we adopt
the explainable convention $+\infty -(+\infty )=+\infty $ for
avoiding ambiguity. Note that the finiteness of $\alpha $\ implies
that dom $g\subset $ dom $h$ and dom $h^{\ast }\subset $ \ dom $g^{\ast }$,
where the effective domain of $\varphi \in \Gamma _{0}$($\textrm{I\!R}^{n})$
is dom $\varphi :=\{x\in \textrm{I\!R}^{n}:\varphi (x)<+\infty \}.$
The function $\varphi \in \Gamma _{0}$($\textrm{I\!R}^{n})$ is polyhedral
convex if it is the sum of the indicator function of a nonempty polyhedral
convex set and the pointwise supremum of a finite collection of affine
functions. Polyhedral DC program is a DC program in which at least one of
the functions $g$ and $h$ is polyhedral convex. Polyhedral DC programming,
which plays a key role in nonconvex programming and global optimization, has
interesting properties (from both a theoretical and an algorithmic point of
view) on local optimality conditions and the finiteness of DCA's convergence.

For $\varphi \in \Gamma _{0}$($\mathrm{I\!R}^{n})$, the subdifferential of $%
\varphi $ at $x_{0}\in $ dom $\varphi ,$ denoted by $\partial \varphi
(x_{0}),$ is defined by 
\begin{equation}
\partial \varphi (x_{0}):=\{y\in \mathrm{I\!R}^{n}:\varphi (x)\geq \varphi
(x_{0})+\langle x-x_{0},y\rangle ,\forall x\in \mathrm{I\!R}^{n}\}.
\label{subdiff}
\end{equation}%
The subdifferential $\partial \varphi (x_{0})$ is a closed convex set, which
generalizes the derivative of $\varphi $ in the sense that $\varphi $ is
differentiable at $x_{0}$ if and only if $\partial \varphi (x_{0})$ is
reduced to a singleton, that is nothing but $\{\bigtriangledown \varphi
(x_{0})\}.$

DC programming investigates the structure of $DC({\mathbb{R}}^{n})$, DC
duality and local and global optimality conditions for DC programs. The
complexity of DC programs clearly lies in the distinction between local and
global solution and, consequently; the lack of verifiable global optimality
conditions.

We have developed necessary local optimality conditions for the primal DC
program $(P_{dc})$, by symmetry those relating to dual DC program\emph{\ }$%
(D_{dc})$\emph{\ }are trivially deduced:%
\begin{equation}
\partial h(x^{\ast })\cap \partial g(x^{\ast })\neq \emptyset
\label{critical point}
\end{equation}%
(such a point $x^{\ast }$ is called critical point of $g-h$ or (\ref%
{critical point}) a generalized Karusk-Kuhn-Tucker (KKT) condition for $%
(P_{dc})$), and 
\begin{equation}
\emptyset \neq \partial h(x^{\ast })\subset \partial g(x^{\ast }).
\label{subdifferential_inclusion}
\end{equation}%
The condition (\ref{subdifferential_inclusion}) is also sufficient (for
local optimality) in many important classes of DC programs. In particular it
is sufficient for the next cases quite often encountered in practice:

\begin{itemize}
\item In polyhedral DC programs with $h$ being a polyhedral convex function.
In this case, if $h$ is differentiable at a critical point $x^{\ast }$, then 
$x^{\ast }$ is actually a local minimizer for $(P_{dc})$. Since a convex
function is differentiable everywhere except for a set of measure zero, one
can say that a critical point $x^{\ast }$ is almost always a local minimizer
for $(P_{dc})$.

\item In case the function $f$ is locally convex at $x^{\ast }$. Note that,
if $h$ is polyhedral convex, then $f=g-h$ is locally convex everywhere $h$
is differentiable.
\end{itemize}

The transportation of global solutions between $(P_{dc})$ and $(D_{dc})$ is
expressed by: 
\begin{equation}
\lbrack \bigcup\limits_{y^{\ast }\in \mathcal{D}}\,\partial g^{\ast
}(y^{\ast })]\subset \mathcal{P}\text{ },\text{ }[\bigcup\limits_{x^{\ast
}\in \mathcal{P}}\,\partial h(x^{\ast })]\subset \mathcal{D}
\label{global transportation}
\end{equation}%
The first (second) inclusion becomes equality if the function $h$ (resp. $%
g^{\ast }$) is subdifferentiable on $\mathcal{P}$ (resp. $\mathcal{D}$).
They show that solving a DC program implies solving its dual. Note also
that, under technical conditions, this transportation also holds for local
solutions of $(P_{dc})$\emph{\ }and $(D_{dc})$. (\cite%
{lethi-website,lethithesis,PLT97,COCONUT,SIOPT2,LTP05,Acta,PLT98} and
references quoted therein).

\noindent \textbf{Philosophy of DCA:} DCA is based on local optimality
conditions and duality in DC programming. The main original idea of DCA is
simple, it consists in approximating a DC\ program by a sequence of convex
programs: each iteration $k$ of DCA approximates the concave part $-h$ by
its affine majorization (that corresponds to taking $y^{k}\in \partial
h(x^{k}))$ and minimizes the resulting convex function.

The generic DCA scheme can be described as follows:

\noindent\textbf{DCA scheme}

\noindent \textbf{Initialization: }Let $x^{0}\in \mathrm{I\!R}^{n}$ be a
guess, set $k:=0.$

\noindent \textbf{Repeat }

\begin{itemize}
\item Calculate some$\ y^{k}\in \partial h(x^{k})$

\item Calculate $x^{k+1}\in \arg \min \{g(x)-[h(x^{k})+\langle
x-x^{k},y^{k}\rangle ]:x\in \mathrm{I\!R}^{n}\}\quad (P_{k})$

\item Increasing $k$ by $1$
\end{itemize}

\noindent \textbf{Until }convergence of $\{x^{k}\}.$

\noindent Note that $(P_{k})$ is a convex optimization problem and is so far
"easy" to solve.

\noindent Convergence properties of DCA and its theoretical basis can be
found in \cite{lethithesis,PLT97,COCONUT,LTP05,Acta,PLT98,PLT13,ANT14}. For
instance it is important to mention that (for the sake of simplicity we omit
here the dual part of DCA).

\begin{itemize}
\item[i)] DCA is a descent method without linesearch (the sequence\ $%
\{g(x^{k})-h(x^{k})\}$ is\ decreasing) but with global convergence (DCA
converges from any starting point).

\item[ii)] If $g(x^{k+1})-h(x^{k+1})=g(x^{k})-h(x^{k})$, then $x^{k}$ is a
critical point of $g-h$. In such a case, DCA terminates at $k$-th iteration.

\item[iii)] If the optimal value $\alpha \ $of problem $(P_{dc})$ is finite
and the infinite sequence $\{x^{k}\}\ $is bounded, then every limit point $%
x^{\ast }$ of the sequence $\{x^{k}\}$\ is a critical point of $g$ $-$ $h$.

\item[iv)] DCA has a \textit{linear convergence} for DC programs.

\item[v)] DCA has a \textit{finite convergence} for polyhedral DC programs.
Moreover, if $h$ is polyhedral and $h$ is differentiable at $x^{\ast }$ then 
$x^{\ast }$ is a local optimizer of $(P_{dc})$.
\end{itemize}

vi) In DC programming with subanalytic data, the whole sequence\textbf{\ }$%
\{x^{k}\}\ $generated by DCA converges and DCA's rate convergence is stated.

It is worth mentioning that the construction of DCA involves DC components $%
g $ and $h$ but not the function $f$ itself. Hence, for a DC program, each
DC decomposition corresponds to a different version of DCA. Since a DC
function $f$ \ has infinitely many DC decompositions which have crucial
implications on the qualities (speed of convergence, robustness, efficiency,
globality of computed solutions,\ldots ) of DCA, the search of a
\textquotedblleft good\textquotedblright\ DC decomposition is important from
an algorithmic point of view. For a given DC program, the choice of \
optimal DC~decompositions is still open. Of course, this depends strongly on
the very specific structure of the problem being considered. In order to
tackle the large-scale setting, one tries in practice to choose\textbf{\ }$g$%
\textbf{\ }and\textbf{\ }$h$\textbf{\ }such that sequences\textbf{\ }$%
\{x^{k}\}$\textbf{\ }and\textbf{\ }$\{y^{k}\}$\textbf{\ }can be easily
calculated\textbf{, }i.e., either they are in an explicit form or their%
\textbf{\ }computations are inexpensive.\textbf{\ }Very often in practice,
the sequence\textbf{\ }$\{y^{k}\}$\textbf{\ }is explicitly computed because
the calculation of a subgradient of\textbf{\ }$h$\textbf{\ }can be\textbf{\ }%
explicitly obtained by using the usual rules for calculating subdifferential
of convex functions\textbf{. }But the solution of the convex program\textbf{%
\ }$(P_{k})$\textbf{, }if not\textbf{\ }explicit, should be achieved by
efficient algorithms well-adapted to its special structure, in order to
handle the large-scale setting\textbf{. }

How to develop an efficient algorithm based on the generic DCA scheme for a
practical problem is thus a sensible question to be studied. Generally, the
answer depends on the specific structure of the problem being considered.
The solution of a nonconvex program $(P_{dc})$ by DCA must be composed of
two stages: the search of an \textit{appropriate} DC decomposition of $f$
and that of a \textit{good} initial point.

DC programming and DCA have been successfully applied for modeling and
solving many and various\ nonconvex programs from different fields of
Applied Sciences, especially in machine learning (see also the more complete
list of references in \cite{lethi-website}). Note that with appropriate DC
decompositions and suitably equivalent DC reformulations, DCA permits to
recover most of standard methods in convex and nonconvex programming as
special cases. In particular, DCA is a global algorithm (i.e. providing
global solutions) when applied to convex programs recast as DC programs and
therefore DC programming and DCA can be used to build efficiently customized
algorithms for solving convex programs generated by DCA itself.

For a complete study of DC programming and DCA the reader is referred to (%
\cite{lethithesis,Acta,COCONUT,PLT97,LTP05,PLT98,PLT13,ANT14} and the
references quoted therein).


\section{DC approximation approaches: consistency results}

\label{approxim}

We focus on the sparse optimization problem with $\ell _{0}$-norm in the
objective function, called the $\ell _{0}$-problem, that takes the form 
\begin{equation}
\min \left\{ F(x,y)=f(x,y)+\lambda \Vert x\Vert _{0}:(x,y)\in K\right\} ,
\label{prob:l0}
\end{equation}%
where $\lambda $ is a positive parameter, $K$ is a convex set in $\mathbb{R}%
^{n}\times \mathbb{R}^{m}$ and $f$ is a finite DC function on $\mathbb{R}%
^{n}\times \mathbb{R}^{m}$. Suppose that $f$ has a DC decomposition 
\begin{equation}
f(x,y)=g(x,y)-h(x,y)\quad \forall (x,y)\in \mathbb{R}^{n}\times \mathbb{R}%
^{m},  \label{eqn:fDC}
\end{equation}%
where $g,h$ are finite convex functions on $\mathbb{R}^{n}\times \mathbb{R}%
^{m}$. Through the paper, for a DC function $f:=g-h$, $\partial f(x,y)$
stands for the set $\partial g(x,y)-\partial h(x,y)$. 
More precisely, the notation $(\overline{x},\overline{y})\in \partial f(x,y)$
means that $(\overline{x},\overline{y})=(x_{g},y_{g})-(x_{h},y_{h})$ for
some $(x_{g},y_{g})\in \partial g(x,y)$, $(x_{h},y_{h})\in \partial h(x,y)$.

Define the step function $s:\mathbb{R}\rightarrow \mathbb{R}$ by $s(t)=1$
for $t\neq 0$ and $s(t)=0$ otherwise. Then $\Vert x\Vert
_{0}=\sum_{i=1}^{n}s(x_{i})$. The idea of approximation methods is to
replace the discontinuous step function by a continuous approximation $%
r_{\theta }$, where $\theta >0$ is a parameter controling the tightness of
approximation. This leads to the approximate problem of the form 
\begin{equation}
\min \left\{ F_{r_{\theta }}(x,y)=f(x,y)+\lambda \sum_{i=1}^{n}r_{\theta
}(x_{i}):(x,y)\in K\right\} .  \label{prob:app}
\end{equation}

\begin{assumption}
\label{assump} $\{r_\theta\}_{\theta > 0}$ is a family of functions $\mathbb{%
R} \to \mathbb{R}$ satisfying the following properties:

\begin{itemize}
\item[i)] $\lim_{\theta \to +\infty}r_\theta (t) = s(t)$, $\forall t \in 
\mathbb{R}$. 

\item[ii)] For any $\theta >0$, $r_{\theta }$ is even, i.e. $r_{\theta
}(t)=r_{\theta }(|t|)~\forall t\in \mathbb{R})$ and $r_{\theta }$ is
increasing on $[0,+\infty )$.

\item[iii)] For any $\theta >0$, $r_{\theta }$ is a DC function which can be
represented as 
\begin{equation*}
r_{\theta }(t)=\varphi _{\theta }(t)-\psi _{\theta }(t)\quad t\in \mathbb{R},
\end{equation*}%
where $\varphi _{\theta },\psi _{\theta }$ are finite convex functions on $%
\mathbb{R}$.

\item[iv)] $t \mu \geq 0$ $\forall t\in \mathbb{R},\mu\in \partial r_{\theta
}(t)$.where $\partial r_{\theta }(t)=\{u-v:u\in \partial \varphi
_{\theta}(t),v\in \partial \psi _{\theta }(t)\}$.

\item[v)] For any $a\leq b$ and $0\notin \lbrack a,b]$: $\lim\limits_{\theta
\rightarrow +\infty }\sup \left\{ |z|:z\in \partial r_{\theta }(t),t\in
\lbrack a,b]\right\} =0.$
\end{itemize}
\end{assumption}

First of all, we observe that by assumption ii) above, we get another
equivalent form of \eqref{prob:app} 
\begin{equation}
\min_{(x,y,z)\in \Omega _{1}}\overline{F}_{r_{\theta
}}(x,y,z):=f(x,y)+\lambda \sum_{i=1}^{n}{r_{\theta }}(z_{i}),
\label{prob:app-z}
\end{equation}%
where 
\begin{equation*}
\Omega _{1}=\{(x,y,z):(x,y)\in K,|x_{i}|\leq z_{i}~~\forall i=1,\dots ,n\}.
\end{equation*}%
Indeed, \eqref{prob:app} and \eqref{prob:app-z} are equivalent in the
following sense.

\begin{proposition}
\label{prop:eqv} A point $(x^*,y^*) \in K$ is a global (resp. local)
solution of the problem \eqref{prob:app} if and only if $(x^*,y^*,|x^*|)$ is
a global (resp. local) solution of the problem \eqref{prob:app-z}. Moreover,
if $(x^*,y^*,z^*)$ is a global solution of \eqref{prob:app-z} then $%
(x^*,y^*) $ is a global solution of \eqref{prob:app}.
\end{proposition}

\begin{proof}
Since $r_\theta$ is an increasing function on $[0,+\infty)$, we have
\[
\overline{F}_{r_\theta}(x,y,z) \ge \overline{F}_{r_\theta}(x,y,|x|) = F_{r_\theta}(x,y) \quad \forall (x,y,z) \in \Omega_1.
\]
Then the conclusion concerning global solutions is trivial. The result on local solutions also follows by remarking that if $(x,y,z) \in B((x^*,y^*,z^*),\delta)$~\footnote{$B(u^*,\delta)$ stands for the set of vectors $u \in \mathbb{R}^d$  such that $\|u-u^*\| < \delta$}  then $(x,y) \in B((x^*,y^*),\delta)$,
and if $(x,y) \in B((x^*,y^*),\frac{\delta}{2})$   then $(x,y,|x|) \in B((x^*,y^*,|x^*|),\delta)$.
\end{proof}

In standard nonconvex approximation approaches to $\ell _{0}$-problem, all
the proposed approximation functions $r_{\theta }$ are even and concave
increasing on $[0,+\infty )$ (see Table \ref{tab:app-form} below) and the
approximate problems were often considered in the form (\ref{prob:app-z}).
Here we study the general case where $r_{\theta }$ is a DC function and
consider both problems (\ref{prob:app}) and (\ref{prob:app-z}) in order to
exploit the nice effect of DC decompositions of a DC program.

Now we show the link between the original problem (\ref{prob:l0}) and the
approximate problem \eqref{prob:app}. This result gives \textit{a
mathematical} \textit{foundation }of approximation methods.

\begin{theorem}
\label{thm:glo} Let $\mathcal{P},\mathcal{P}_\theta$ be the solution sets of
the problem \eqref{prob:l0} and \eqref{prob:app} respectively.

\begin{enumerate}
\item[i)] Let $\{\theta _{k}\}$ be a sequence of nonnegative numbers such
that $\theta _{k}\rightarrow +\infty $ and $\{(x^{k},y^{k})\}$ be a sequence
such that $(x^{k},y^{k})\in \mathcal{P}_{\theta _{k}}$ for any $k$. If $%
(x^{k},y^{k})\rightarrow (x^{\ast },y^{\ast })$, then $(x^{\ast },y^{\ast
})\in \mathcal{P}$.

\item[ii)] If $K$ is compact, then for any $\epsilon >0$ there is $\theta
(\epsilon )>0$ such that 
\begin{equation*}
\mathcal{P}_{\theta }\subset \mathcal{P}+B(0,\epsilon )\quad \forall \theta
\geq \theta (\epsilon ).
\end{equation*}

\item[iii)] If there is a finite set $\mathcal{S}$ such that $\mathcal{P}%
_{\theta }\cap \mathcal{S}\neq \emptyset ~\forall \theta >0$, then there
exists $\theta _{0}\geq 0$ such that 
\begin{equation*}
\mathcal{P}_{\theta }\cap \mathcal{S}\subset \mathcal{P}\quad \forall \theta
\geq \theta _{0}.
\end{equation*}
\end{enumerate}
\end{theorem}

\begin{proof}
i) Let $(x,y)$ be arbitrary in $K$. For any $k$, since $(x^k,y^k) \in \mathcal{P}_{\theta_k}$, we have
\begin{equation}\label{eqn:thm1}
f(x,y) + \lambda \sum_{i=1}^n r_{\theta_k}(x_i) \ge f(x^k,y^k) + \lambda \sum_{i=1}^n r_{\theta_k}(x^k_i).
\end{equation}
By Assumption \ref{assump} ii), if $x^*_i = 0$, we have
\[
\liminf_{k\to +\infty} r_{\theta_k}(x^k_i) \ge \liminf_{k\to +\infty} r_{\theta_k}(0) = 0.
\]
If $x^*_i \ne 0$, there exist $a_i \le b_i$ and $k_i \in \mathbb{N}$ such that $0 \ne [a_i, b_i]$ and $x^k_i \in [a_i,b_i]$ for all $k \ge k_i$. Then we have
\[
|r_{\theta_k}(x^k_i) - s(x^*_i)| \le \max\left\{|r_{\theta_k}(a_i) - s(a_i)|,|r_{\theta_k}(b_i) - s(b_i)| \right\} \quad \forall k \ge k_i.
\]
Since $\lim_{k \to +\infty} r_{\theta_k}(a_i) = s(a_i)$ and $\lim_{k \to +\infty} r_{\theta_k}(b_i) = s(b_i)$, we have $\lim_{k \to +\infty} r_{\theta_k}(x^k_i) = s(x^*_i)$. Note that $f$ is continuous, taking $\liminf$ of both sides of \eqref{eqn:thm1}, we get
\[
f(x,y) + \lambda \sum_{i=1}^n s(x_i) \ge f(x^*,y^*) + \lambda \sum_{i=1}^n \liminf_{k \to \infty} r_{\theta_k}(x^k_i) \ge f(x^*,y^*) + \lambda \sum_{i=1}^n s(x^*_i).
\]
Thus, $F(x,y) \ge F(x^*,y^*)$ for any $(x,y) \in K$, or $(x^*,y^*) \in \mathcal{P}$.

ii) We assume by contradiction that there exists $\epsilon > 0$ and a sequence $\{\theta_k\}$ such that $\theta_k \to +\infty$, and for any $k$ there is $(x^k,y^k) \in \mathcal{P}_{\theta_k} \backslash (\mathcal{P} + B(0,\epsilon))$. Since $\{(x^k,y^k)\} \subset K$ and $K$ is compact, there exists a subsequence $\{(x^{k_l},y^{k_l})\}$ of $\{(x^k,y^k)\}$ converges to a point $(x^*,y^*) \in K$. By i), we have $(x^*,y^*) \in \mathcal{P}$. However, $\{(x^{k_l},y^{k_l})\} \subset K \backslash (\mathcal{P} + B(0,\epsilon))$ that is a closed set, so $(x^*,y^*) \in K \backslash (\mathcal{P} + B(0,\epsilon))$. This contradicts the fact that $(x^*,y^*) \in \mathcal{P}$. 

iii) Assume by contradiction that there is a sequence $\{\theta_k\}$ such that $\theta_k \to +\infty$, 
and for any $k$ there is $(x^k,y^k) \in (\mathcal{P}_{\theta_k} \cap \mathcal{S}) \backslash \mathcal{P}$. Since $\mathcal{S}$ is finite, we can extract a subsequence such that $(x^{k_l},y^{k_l}) = (\overline{x},\overline{y})~\forall l$. Then we have $(\overline{x},\overline{y}) \notin \mathcal{P}$. This contradicts the fact that    $(\overline{x},\overline{y}) \in \mathcal{P}$ following i).  
\end{proof}

\begin{remark}
The assumption that $r_{\theta }$ is an even function is not needed for
proving this theorem.\ More precisely, the theorem still holds when the
assumption ii) is replaced by \textquotedblright for any $\theta >0$, $%
r_{\theta }$ is decreasing on $(-\infty ,0]$ and is increasing on $%
[0,+\infty ).$ For the zero-norm, since the step function is even, it is
natural to consider its approximation $r_{\theta }$ as an even function.
\end{remark}

Theorem \ref{thm:glo} shows that any optimal solution of the approximate
problem (\ref{prob:app}) is in a $\epsilon -$neighboohord of an optimal
solution to the original problem (\ref{prob:l0}), and the tighter
approximation of $\ell _{0}$-norm is, the better approximate solutions are.
Moreover, if there is a finite set $\mathcal{S}$ such that $\mathcal{P}%
_{\theta }\cap \mathcal{S}\neq \emptyset ~\forall \theta >0,$ then any
optimal solution of the approximate problem (\ref{prob:app}) contained in $%
\mathcal{S}$ solves also the problem (\ref{prob:l0}). By considering the
equivalent problem (\ref{prob:app-z}), we show in the following Corollary
that such a set $\mathcal{S}$ exists in several contexts of applications
(for instance, in feature selection in SVM).

\begin{corollary}
\label{crl:1} Suppose that $r$ is concave on $[0,+\infty )$, $K$ is a
polyhedral convex set having at least a vertex and $f$ is concave, bounded
below on $K$. Then $\Omega _{1}$ defined in \eqref{prob:app-z} is also a
polyhedral convex set having at least a vertex\textbf{. }Let $\mathcal{V}$
be the vertex set of $\Omega _{1}$ and 
\begin{equation*}
\overline{\mathcal{P}}_{\theta }=\left\{ (x,y):\exists z\in \mathbb{R}^{n}%
\text{ s.t. }(x,y,z)\in \mathcal{V}\text{ is a global solution of %
\eqref{prob:app-z}}\right\} .
\end{equation*}%
Then $\overline{\mathcal{P}}_{\theta }\neq \emptyset ~\forall \theta >0$ and
there exists $\theta _{0}>0$ such that $\overline{\mathcal{P}}_{\theta
}\subset \mathcal{P}$,$~\forall \theta \geq \theta _{0}$.
\end{corollary}

\begin{proof}
By the assumptions, we have $\overline{F}_{r_\theta}$ is  concave, bounded below on $\Omega_1$, so $\overline{\mathcal{P}}_\theta \ne \emptyset ~\forall \theta > 0$. 
Let $\mathcal{S} = \{(x,y): (x,y,z) \in \mathcal{V} \text{ for some } z \in \mathbb{R}^n\}$. By Proposition \ref{prop:eqv}, we have $\overline{\mathcal{P}}_\theta \subset \mathcal{P}_\theta \cap \mathcal{S} ~\forall \theta > 0$. Since $\mathcal{V}$ is finite, so is $\mathcal{S}$. The property  iii) of Theorem \ref{thm:glo} implies the existence of $\theta_0 > 0$ such that
\[\overline{\mathcal{P}}_\theta \subset \mathcal{P}_\theta \cap \mathcal{S} \subset \mathcal{P}~ \forall \theta \ge \theta_0.
\]
\end{proof}

Note that the consistency between the solution of the approximate problem
and the original problem have been carried out in \cite{Bradley98b} (resp. 
\cite{Rinaldi10}) for the case where $f$ is concave, bounded below on the
polyhedral convex set $K$ and $r$ is the exponential approximation defined
in Table \ref{tab:app-form} below (resp. $r$ is the logarithm function
and/or $\ell _{p}$-norm ($p<1$)). Here, besides general results carried out
in Theorem \ref{thm:glo}, our Corollary \ref{crl:1} gives a much stronger
result than those in \cite{Bradley98b,Rinaldi10} where they only ensure that 
$\overline{\mathcal{P}}_{\theta }\cap \mathcal{P}\neq \emptyset ~\forall
\theta \geq \theta _{0}$.

Observing that the approximate problem is still nonconvex for which, in
general, only local algorithms are available, we are motivated by the study
of the consistency between local minimizers of the original and approximate
problems. For this purpose, first, we need to describe characteristics of
local solutions of these problems.

\begin{proposition}
i) A point $(x^{\ast },y^{\ast })\in K$ is a local optimum of the problem %
\eqref{prob:l0} if and only if $(x^{\ast },y^{\ast })$ is a local optimum of
the problem 
\begin{equation}
\min \{f(x,y):(x,y)\in K(x^{\ast })\},  \label{prob:cut}
\end{equation}%
where $K(x^{\ast })=\{(x,y)\in K:x_{i}=0~\forall i\notin supp(x^{\ast })\}$.

ii) If $(x^{\ast },y^{\ast })\in K$ is a local optimum of the problem %
\eqref{prob:l0} then 
\begin{equation}
\langle \overline{x}^{\ast },x-x^{\ast }\rangle +\langle \overline{y}^{\ast
},y-y^{\ast }\rangle \geq 0\quad \forall (x,y)\in K(x^{\ast }),
\label{cond:org}
\end{equation}%
for some $(\overline{x}^{\ast },\overline{y}^{\ast })\in \partial f(x^{\ast
},y^{\ast })$.
\end{proposition}

\begin{proof}
i) The forward implication is obvious, we only need to prove the backward one. Assume that $(x^*,y^*)$ is a local solution of the problem \eqref{prob:cut}. 
There exists a neighbourhood $\mathcal{V}$ of $(x^*,y^*)$ such that 
\[
supp(x^*) \subset supp(x) \quad \text{and} \quad |f(x,y) - f(x^*,y^*)| < \lambda \quad \forall (x,y) \in \mathcal{V},
\]
and 
\[
f(x^*,y^*) \le f(x,y) \quad \forall (x,y) \in \mathcal{V}\cap K(x^*).
\]
For any $(x,y) \in \mathcal{V}\cap K$,   two cases occur:

- If $(x,y) \in K(x^*)$, then $\|x\|_0 = \|x^*\|_0$ and $f(x^*,y^*) \le f(x,y)$.

- If $(x,y) \notin K(x^*)$, then $\|x^*\|_0 \le \|x\|_0 - 1$ and $f(x^*,y^*) < f(x,y) + \lambda$.\\
In   both cases, we have $f(x^*,y^*) + \lambda \|x^*\|_0 \le f(x,y) + \lambda \|x\|_0$. Thus, $(x^*,y^*)$ is a local solution of the problem \eqref{prob:l0}.\\
ii) Since $f = g-h$ is a DC function, \eqref{prob:cut} is a DC program. Therefore, the  necessary local condition of the problem \eqref{prob:cut} can be stated by
\[
0 \in \partial (g + \chi_{K(x^*)})(x^*,y^*) - \partial h(x^*,y^*),
\]
or equivalently, there exists $(\overline{x}^*,\overline{y}^*) \in \partial f(x^*,y^*)$ such that
\[
- (\overline{x}^*,\overline{y}^*) \in \partial \chi_{K(x^*)}(x^*,y^*)
\Leftrightarrow  \langle \overline{x}^*,x - x^* \rangle + \langle \overline{y}^*,y - y^* \rangle \ge 0 \quad \forall (x,y) \in K(x^*),
\]
\end{proof}

As for the characteristics of local solutions of the problem (\ref{prob:app}

), we follow the condition (\ref{critical point}) above for a DC program.
Writing the problem \eqref{prob:app} in form of a DC program 
\begin{equation}
\min_{x,y}\{F_{r_{\theta }}(x,y):=G(x,y)-H(x,y)\},  \label{eqn:DC}
\end{equation}%
with 
\begin{equation}
G(x,y)=\chi _{K}(x,y)+g(x,y)+\lambda \sum_{i=1}^{n}\varphi _{\theta }(x_{i}),%
\text{ }H(x,y)=h(x,y)+\lambda \sum_{i=1}^{n}\psi _{\theta }(x_{i}).
\label{decomp0}
\end{equation}
Then, for a point $(x^{\ast },y^{\ast })\in K,$ the necessary local
optimality condition (\ref{critical point}) can be expressed as 
\begin{equation*}
0\in \partial G(x^{\ast },y^{\ast })-\partial H(x^{\ast },y^{\ast }),
\end{equation*}%
which is equivalent to 
\begin{equation}
\langle \overline{x}^{\ast },x-x^{\ast }\rangle +\langle \overline{y}^{\ast
},y-y^{\ast }\rangle +\langle \overline{z}^{\ast },x-x^{\ast }\rangle \geq
0\quad \forall (x,y)\in K,  \label{cond:app}
\end{equation}%
for some $(\overline{x}^{\ast },\overline{y}^{\ast })\in \partial f(x^{\ast
},y^{\ast })$ and $\overline{z}_{i}^{\ast }\in \lambda \partial r_{\theta
}(x_{i}^{\ast })~~\forall i=1,\dots ,n$.

Now we are able to state consistency results of local optimality.

\begin{theorem}
\label{thm:local} Let $\mathcal{L}$ and $\mathcal{L}_\theta$ be the sets of $%
(x,y) \in K$ satisfying the conditions \eqref{cond:org} and \eqref{cond:app}
respectively.

\begin{enumerate}
\item[i)] Let $\{\theta _{k}\}$ be a sequence of nonnegative numbers such
that $\theta _{k}\rightarrow +\infty $ and $\{(x^{k},y^{k})\}$ be a sequence
such that $(x^{k},y^{k})\in \mathcal{L}_{\theta _{k}},\forall k$. If $%
(x^{k},y^{k})\rightarrow (x^{\ast },y^{\ast })$, we have $(x^{\ast },y^{\ast
})\in \mathcal{L}$.

\item[ii)] If $K$ is compact then, for any $\epsilon > 0$, there is $%
\theta(\epsilon) > 0$ such that 
\begin{equation*}
\mathcal{L}_\theta \subset \mathcal{L} + B(0,\epsilon) \quad \forall \theta
\ge \theta(\epsilon).
\end{equation*}

\item[iii)] If there is a finite set $\mathcal{S}$ such that $\mathcal{L}%
_\theta \cap \mathcal{L} \ne \emptyset,\forall \theta > 0$, then there
exists $\theta_0\ge 0$ such that 
\begin{equation*}
\mathcal{L}_\theta \cap \mathcal{S} \subset \mathcal{L} \quad \forall \theta
\ge \theta_0.
\end{equation*}
\end{enumerate}
\end{theorem}

\begin{proof}
i) By definition, there is a sequence $\{(\overline{x}^k,\overline{y}^k,\overline{z}^k)\}$ such that for all $k=1,2,\dots$
\[
(\overline{x}^k,\overline{y}^k) \in \partial f(x^k,y^k), \text{ and } \overline{z}^k_i \in \lambda \partial r_{\theta_k}(x^k_i)~~ i=1,\dots,n,
\]
\begin{equation}\label{eqn:inq_thm2}
\langle \overline{x}^k,x - x^k \rangle + \langle \overline{y}^k,y - y^k \rangle + \langle \overline{z}^k,x - x^k \rangle \ge 0 \quad \forall (x,y) \in K.
\end{equation}
For $k=1,2,\dots$, we have
\[
(\overline{x}^k,\overline{y}^k) = (x^k_g,y^k_g) - (x^k_h,y^k_h),
\]
where $(x^k_g,y^k_g) \in \partial g(x^k,y^k)$, and $(x^k_h,y^k_h) \in \partial h(x^k,y^k)$. 

Since $\{(x^k,y^k)\}$ converges to $(x^*,y^*)$, there is $k_0 \in \mathbb{N}$ and a compact set $\mathcal{S} \subset \mathbb{R}^n\times \mathbb{R}^m$ such that $(x^k,y^k) \in \mathcal{S}, ~\forall k\ge k_0$. It  follows by Theorem 24.7 (\cite{Rockafellar}) that $\partial g(\mathcal{S}):=\cup_{x \in\mathcal S}\partial g(x)$ and $\partial h(\mathcal{S}):=\cup_{x \in\mathcal S}\partial h(x)$ are compact sets. Thus, there is an infinite set $\mathcal{K}\subset\mathbb{N}$ such that the sequence $\{(x^k_g,y^k_g)\}_{k \in \mathcal{K}}$ converges to a point $(x^*_g,y^*_g) \in \partial g(\mathcal S)$ and the sequence $\{(x^k_h,y^k_h)\}_{k \in \mathcal{K}}$ converges to a point $(x^*_h,y^*_h) \in \partial h(\mathcal S)$. By Theorem 24.4 (\cite{Rockafellar}), we have $(x^*_g,y^*_g) \in \partial g(x^*,y^*)$ and $(x^*_h,y^*_h) \in \partial h(x^*,y^*)$. Therefore, the sequence $\{(\overline{x}^k,\overline{y}^k)\}_{k \in \mathcal{K}}$ converges to $(\overline{x}^*,\overline{y}^*) = (x^*_g,y^*_g) - (x^*_h,y^*_h) \in \partial f(x^*,y^*)$.

By Assumption \ref{assump} iv), we have $\overline{z}^k_i x^k_i \ge 0 ~\forall i,k$. Moreover, for any $i \in supp(x^*)$, there exist $a_i \le b_i$ and $k_i \in \mathbb{N}$ such that $0 \notin [a_i, b_i]$ and $x^k_i \in [a_i,b_i]$ for all $k \ge k_i$. By Assumption \ref{assump} v), we deduce that $ \overline{z}^k_i \to 0$ as $k \to +\infty$. 

For arbitrary $(x,y) \in K(x^*)$, \eqref{eqn:inq_thm2} implies that
\begin{eqnarray*}
\langle \overline{x}^k,x - x^k \rangle + \langle \overline{y}^k,y - y^k \rangle & \ge & \sum_{i \notin supp(x^*)} \overline{z}^k_i x^k_i  - \sum_{i \in supp(x^*)} \overline{z}^k_i (x_i - x^k_i) \\
& \ge &  - \sum_{i \in supp(x^*)} \overline{z}^k_i (x_i - x^k_i) \quad \forall k.
\end{eqnarray*}
Taking $k\in \mathcal{K}, k \to +\infty$, we get
\[
\langle \overline{x}^*,x - x^* \rangle + \langle \overline{y}^*,y - y^* \rangle \ge 0 \quad \forall (x,y) \in K(x^*).
\]
Thus, $(x^*,y^*) \in \mathcal{L}$.

ii) and iii) are proved similarly as in Theorem \ref{thm:glo}.
\end{proof}


\section{DC approximation functions}

\label{Sect.compare}

First, let us mention, in chronological order, the approximation functions
proposed in the literature in different contexts, but we don't indicate the
related works concerning algorithms using these approximations). The first
was concave exponential approximation proposed in \cite{Bradley-Mangasarian}
in the context of feature selection in SVM, and $\ell _{p}$-norm with $0<p<1$
for sparse regression (\cite{Fu98}). Later, the $\ell _{p}$-norm with $p<0$
was studied in \cite{Rao99} for sparse signal recovery, and then the
Smoothly Clipped Absolute Deviation (SCAD) \cite{FAN01} in the context of
regression, the logarithmic approximation \cite{Weston03} for feature
selection in SVM, and the Capped-$\ell _{1}$ (\cite{Peleg08}) applied on
sparse regression.

A common property of these approximations is they are all even, concave
increasing functions on $[0,+\infty ).$ It is easy to verify that these
function satisfy the conditions in Assumption 1 and so they are particular
cases of our DC approximation $r$. More general DC approximation functions
are also investigated, e.g., PiL (\cite{LT12}) that is a (nonconcave)
piecewise linear function defined inTable \ref{tab:app-form}.

Note that, some of these approximation functions, namely logarithm (log),
SCAD and $\ell _{p}$-norm defined by 
\begin{equation}
Log:\log (|t|+\epsilon ),\,\epsilon >0,\quad \ell _{p}:\mathrm{sign}%
(p)(|t|+\epsilon )^{p},\,0\neq p\leq 1,\epsilon >0;  \label{log-lp}
\end{equation}%
\begin{equation}
SCAD:%
\begin{cases}
\gamma |t| & \text{ if $0\leq |t|\leq \gamma ,$}\quad \frac{(a+1)\gamma ^{2}%
}{2}\text{if $|t|\geq a\gamma $} \\ 
\frac{-t^{2}+2a\gamma |t|-\gamma ^{2}}{2(a-1)} & \text{ if $\gamma
<|t|<a\gamma $} \\ 
& 
\end{cases}%
,a>1,\gamma >0  \label{scad}
\end{equation}%
do not directly approximate $\ell _{0}$-norm. But they become approximations
of $\ell _{0}$-norm if we multiply them by an appropriate factor (which can
be incorporated into the parameter $\lambda $), and add an appropriate term
(such a procedure doesn't affect the original problem). The resulting
approximation forms of these functions are given in Table \ref{tab:app-form}.
We see that $r_{scad}$ is obtained by multiplying the SCAD function by $%
\frac{2}{(a+1)\gamma ^{2}}$ and setting $\theta =\frac{1}{\gamma }$.
Similarly, by taking $\theta =\frac{1}{\epsilon }$, we have 
\begin{equation*}
r_{log}(t)=\frac{\log (|t|+\epsilon )}{\log (1+1/\epsilon )}-\frac{\log
\epsilon }{\log (1+1/\epsilon )},\text{ and }r_{\ell _{p}^{-}}(t)=-\frac{%
(|t|+\epsilon )^{p}}{\epsilon ^{p}}+1.
\end{equation*}%
For using $\ell _{p}$-norm approximation with $0<p<1$, we take $\theta =%
\frac{1}{p}$. Note that 
$\lim_{\theta \rightarrow \infty }|t|^{1/\theta }=s(t)$. To avoid
singularity at $0$, we add a small $\epsilon >0$. In this case, we require $%
\epsilon =\epsilon (\theta )$ satisfying $\lim_{\theta \rightarrow \infty
}\epsilon (\theta )^{1/\theta }=0$ to ensure that $\lim_{\theta \rightarrow
\infty }r_{\ell _{p}^{+}}(t)=s(t)$.

\begin{table}[tbp]
\caption{$\ell_0$-approximation functions $r$ and the first DC decomposition 
$\protect\varphi$. The second DC decomposition is $\protect\psi=\protect%
\varphi-r$.}
\label{tab:app-form}\centering
\begin{tabular}{lll}
\hline
Approximation  & Function $r$ & Function $\varphi$ \\ \hline
Exp (\cite{Bradley-Mangasarian}) & $r_{exp}(t) = 1-e^{-\theta |t|}$ & $%
\theta |t|$ \vspace{0.1cm} \\ 
$\ell_p(0<p<1)$(\cite{Fu98}) & $r_{\ell_p^+}(t) = (|t|+\epsilon)^{1/\theta}$
& $\frac{\epsilon^{1/\theta - 1}}{\theta} |t|$ \vspace{0.1cm} \\ 
$\ell_p(p<0)$(\cite{Rao99}) & $r_{\ell_p^-}(t) = 1- (1+\theta|t|)^p$, $p<0$
& $-p \theta |t|$ \vspace{0.1cm} \\ 
Log (\cite{Weston03}) & $r_{log}(t) = \frac{\log(1+\theta |t|)}{%
\log(1+\theta)}$ & $\frac{\theta}{\log(1+\theta)} |t|$ \vspace{0.1cm} \\ 
SCAD (\cite{FAN01}) & $r_{scad}(t) = 
\begin{cases}
\frac{2\theta}{a+1} |t| & 0 \le |t| \le \frac{1}{\theta} \\ 
\frac{-\theta^2t^2 + 2 a \theta |t| - 1}{a^2-1} & \frac{1}{\theta} < |t| < 
\frac{a}{\theta} \\ 
1 & |t| \ge \frac{a}{\theta}%
\end{cases}%
$ & $\dfrac{2\theta}{a+1} |t|$ \\ 
Capped-$\ell_1$ (\cite{Peleg08}) & $r_{cap}(t) = \min\{1,\theta |t|\}$ & $%
\theta |t|$ \vspace{0.1cm} \\ 
PiL \cite{LT12} & $r_{PiL} = \min\left\{1, \max\left\{0, \frac{\theta |t| - 1%
}{a-1} \right\} \right\}$ & $\frac{\theta}{a-1} \max\left\{\frac{1}{\theta},
|t| \right\}$ \\ \hline
\end{tabular}
\end{table}

All these functions satisfy Assumption \ref{assump} (for proving the
condition iii) of Assumption \ref{assump} we indicate in Table \ref%
{tab:app-form} a DC decomposition of the approximation functions), so the
consistency results stated in Theorems \ref{thm:glo} and \ref{thm:local} are
applicable.

\textbf{Discussion}. Except $r_{PiL}$ that is differentiable at $0$ with $%
r_{PiL}^{\prime }(0)=0$, the other approximations have the right derivative
at $0$ depending on the approximation parameter $\theta $. Clearly the
tightness of each approximation depends on related parameters. Hence, a
suitable way to compare them is using the parameter $\theta $ such that
their right derivatives at $0$ are equal, namely 
\begin{equation*}
\theta _{cap}=\frac{2}{a+1}\theta _{scad}=\theta _{exp}=-p\theta _{\ell
_{p}^{-}}.
\end{equation*}%
In this case, by simple calculation we have 
\begin{equation}
0\leq r_{\ell _{p}^{-}}\leq r_{exp}\leq r_{scad}\leq r_{cap}\leq s.
\label{compare}
\end{equation}%
Comparing $r_{cap}$ and $r_{scad}$ with different values $\theta ,$ we get 
\begin{equation}
\begin{cases}
0\leq r_{scad}\leq r_{cap}\leq s, & \text{if }\frac{2\theta _{scad}}{a+1}%
\leq \theta _{cap} \\ 
0\leq r_{cap}\leq r_{scad}\leq s, & \text{if }\theta _{cap}\leq \frac{\theta
_{scad}}{a}.%
\end{cases}
\label{scad-cap}
\end{equation}%
Inequalities in (\ref{compare}) show that, with the parameter $\theta $ such
that their right derivatives at $0$ are equal, $r_{scad}$ and $r_{cap}$ are
closer to the step function $s$ than $r_{\ell _{p}^{-}}\ $and $r_{exp}$.

As for $r_{log}$ and $r_{\ell _{p}^{+}},$ we see that they tend to $+\infty $
when $t\rightarrow +\infty $, so they have poor approximation for $t$ large.
Whereas, the other approximations are minorants of $s$ and larger $t$ is,
closer to $s$ they are. For easier seeing, we depict these approximations in
Figure \ref{Fig.1}.

\begin{figure}[tbp]
\begin{center}
\includegraphics[scale=0.5]{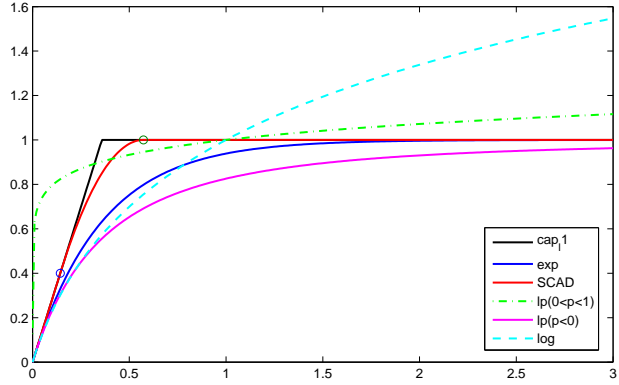}
\end{center}
\caption{Graphs of approximation functions. Except $\ell_p$-norm($0<p<1$)
and PiL, the others have the same derivative at 0. Here $\protect\theta%
_{log} = 10$ for Log, $a = 4$ for SCAD, $p=-2$ for $\ell_p$-norm($p<0$). For 
$\ell_p$-norm($0<p<1$), $\protect\epsilon = 10^{-9}$ and $p=0.2$. For PiL, $%
a = 5$ and $\protect\theta_{PiL} = a \protect\theta_{exp}$.}
\label{Fig.1}
\end{figure}

Now, we give a deeper study on Capped-$\ell _{1}$ approximation. Using exact
penalty techniques related to $\ell _{0}$-norm developed in (\cite%
{Thiao-et-al08,PLT13,LTLP13}) we prove a much stronger result for this
approximation, that is the approximation problem (\ref{prob:app}) is
equivalent to the original problem with appropriate parameters $\theta $
when $K$ is a compact polyhedral convex set (this case quite often occurs in
applications, in particular in machine learning contexts). Furthermore, when 
$K$ is a box, we show (directly, without using the exact penalty techniques)
that the Capped-$\ell _{1}$ approximation problem is equivalent to the
original problem and we compute an exact value $\theta _{0}$ such that the
equivalence holds for all $\theta >\theta _{0}$.


\section{A deeper study on Capped-$\ell _{1}$ approximation problems}

\label{exact}

\subsection{Link between approximation and exact penalty approaches}

Thanks to exact continuous reformulation via penalty techniques, we shall
prove that, with some sparse inducing functions, the approximate problem is
equivalent to the original problem. First of all, let us recall exact
penalty techniques related to $\ell _{0}$-norm (\cite{Thiao-et-al08,PLT13}).

\subsubsection{Continuous reformulation via exact penalty techniques}

Denote by $e$ the vector of ones in the appropriate vector space. We suppose
that $K$ is bounded in the variable $x$, i.e. $K\subset \Pi
_{i=1}^{n}[a_{i},b_{i}]\times \mathbb{R}^{m}$ where $a_{i},b_{i}\in $ $%
\mathbb{R}$ such that $a_{i}\leq 0<b_{i}$ for $i=1,...,n.$ Let $c_{i}:=\max
\{\left\vert x_{i}\right\vert :x_{i}\in \lbrack a_{i},b_{i}]\}=\max
\{\left\vert a_{i}\right\vert ,\left\vert b_{i}\right\vert \}$ for $%
i=1,...,n.$ Define the binary variable $u_{i}\in \left\{ 0,1\right\} $ as%
\begin{equation}
u_{i}=\left\vert x_{i}\right\vert _{0}=%
\begin{cases}
1\text{ if }x_{i}\neq 0 \\ 
0\text{ if }x_{i}=0,%
\end{cases}%
\qquad \forall i=1...n.
\end{equation}

Then (\ref{mainpbl0}) can be reformulated as%
\begin{equation}
\alpha :=\inf \{f(x,y)+\lambda e^{T}u:(x,y)\in K,u\in \{0,1\}^{n},\left\vert
x_{i}\right\vert \leq c_{i}u_{i},\text{ }i=1,...,n\},  \label{l0 obj penal}
\end{equation}

Let $p(u)$ be the penalty function defined by%
\begin{equation}
p(u):=\sum\limits_{i=1}^{n}\min \{u_{i},1-u_{i}\}\text{.}  \label{functp}
\end{equation}%
Then (\ref{mainpbl0}) can be rewritten as

\begin{equation}
\alpha =\inf \{f(x,y)+\lambda e^{T}u:(x,y)\in K,u\in \lbrack
0,1]^{n},\left\vert x_{i}\right\vert \leq c_{i}u_{i},\text{ }i=1,...,n,\text{
}p(u)\leq 0\},  \label{P1}
\end{equation}%
which leads to the corresponding penalized problems $(\tau $ being the
positive penalty parameter)%
\begin{equation}
\alpha (\tau ):=\inf \{f(x,y)+\lambda e^{T}u+\tau p(u):(x,y)\in K,u\in
\lbrack 0,1]^{n},\left\vert x_{i}\right\vert \leq c_{i}u_{i},\text{ }%
i=1,...,n\}.  \label{pP1}
\end{equation}%
It has been shown in \cite{Thiao-et-al08,PLT13} that there is $\tau _{0}\geq
0$ such that for every $\tau >\tau _{0}$ problems (\ref{mainpbl0}) and (\ref%
{pP1}) are equivalent, in the sense that they have the same optimal value
and $(x^{\ast },y^{\ast })\in K$ is a solution of (\ref{mainpbl0}) iff there
is $u^{\ast }\in \left\{ 0,1\right\} ^{n}$ such that $(x^{\ast },y^{\ast
},u^{\ast })$ is a solution of (\ref{pP1}).

It is clear that if the function $f(x,y)$ is a DC function on $K$ then (\ref%
{P1}) is a DC program.

Let us state now the link between the continuous problem (\ref{pP1}) and the
Capped-$\ell _{1}$ approximation problem.

\subsubsection{Link between (\protect\ref{pP1}) and Capped-$\ell _{1}$
approximation problem}

The Capped-$\ell _{1}$ approximation is defined by: 
\begin{equation}
\Psi _{_{\theta }}(x):=\sum_{i=1}^{n}r_{_{cap}}(x_{i}),\forall x=(x_{i})\in 
\mathbb{R}^{n},\text{ with }r_{cap}(t):=\min \{\theta \left\vert
t\right\vert ,1\},\text{ }t\in \mathbb{R}.  \label{psi-norm}
\end{equation}%
\ We will demonstrate that the resulting approximate problem of (\ref%
{mainpbl0}), namely%
\begin{equation}
\beta (\theta ):=\inf \left\{ f(x,y)+\lambda \sum_{i=1}^{n}r_{cap}\left(
x_{i}\right) :(x,y)\in K\right\}  \label{eqn:3}
\end{equation}%
is equivalent to the penalized problem (\ref{pP1}) with suitable values of
parameters $\lambda ,$ $\tau $ and $\theta $.

Let $M = \max\{c_i:i=1,\dots,n\}$, consider the problem (\ref{pP1}) in the
form 
\begin{equation}
\alpha (\tau ):=\inf \{f(x,y)+\lambda e^{T}u+\tau p(u):(x,y)\in K,u\in
\lbrack 0,1]^{n},|x_{i}|\leq Mu_{i},~i=1,\dots ,n\}.  \label{eqn:1}
\end{equation}%
Let $\varsigma :\mathbb{R\rightarrow R}$ be the function defined by $%
\varsigma (t)=\min \{t,1-t\}.$ Then $p(u)=\sum_{i=1}^{n}\varsigma (u_{i})$
and the problem \eqref{eqn:1} can be rewritten as%
\begin{equation}
\alpha (\tau ):=\inf \left\{ f(x,y)+\lambda \sum_{i=1}^{n}\left( u_{i}+\frac{%
\tau }{\lambda }\varsigma (u_{i})\right) :(x,y)\in K,\frac{|x_{i}|}{M}\leq
u_{i}\leq 1,~i=1,\dots ,n\right\} ,
\end{equation}%
or again 
\begin{equation}
\alpha (\tau ):=\inf \left\{ f(x,y)+\lambda \sum_{i=1}^{n}\pi \left(
u_{i}\right) :(x,y)\in K,\frac{|x_{i}|}{M}\leq u_{i}\leq 1,~i=1,\dots
,n\right\}  \label{eqn:2}
\end{equation}%
where $\pi :\mathbb{R\rightarrow R}$ be the function defined by $\pi (t):=t+%
\frac{\tau }{\lambda }\varsigma (t).$

\begin{proposition}
\label{cap} Let $\theta :=\frac{\tau +\lambda }{\lambda M}.$ For all $\tau
\geq \lambda $ problems \eqref{eqn:2} and \eqref{eqn:3} are equivalent in
the following sense: $(x^{\ast },y^{\ast })$ is an optimal solution of %
\eqref{eqn:3} iff $(x^{\ast },y^{\ast },u^{\ast })$ is an optimal solution
of \eqref{eqn:2}, where $u_{i}^{\ast }\in \left\{ \frac{|x_{i}^{\ast }|}{M}%
,1\right\} \ $such that $\pi (u_{i}^{\ast })=r_{cap} (x_{i}^{\ast })$ for $%
i=1,\dots ,n.$ Moreover, $\alpha (\tau )=\beta (\theta )$.
\end{proposition}

\begin{proof} If $(x^{\ast },y^{\ast },u^{\ast })$ is an optimal solution of %
\eqref{eqn:2}, then $u_{i}^{\ast }$ is an optimal solution of the following
problem, for every $i=1,\dots ,n$ 
\begin{equation}
\ \min \left\{ \pi (u_{i}):\frac{|x_{i}^{\ast }|}{M}\leq u_{i}\leq 1\right\}
.  \label{minu}
\end{equation}%
Since $\varsigma$ is a concave function, so is $\pi$. Consequently

\begin{eqnarray*}
\min \left\{ \pi (u_{i}):\frac{|x_{i}^{\ast }|}{M}\leq u_{i}\leq 1\right\}
=\min \left\{ \pi \left( \frac{|x_{i}^{\ast }|}{M}\right) ,\pi (1)\right\}
 =\min \left\{ \left( 1+\frac{\tau }{\lambda }\right) 
\frac{|x_{i}^{\ast }|}{M},1\right\} =r_{cap}(x_{i}^{\ast }).
\end{eqnarray*}%

For an arbitrary $(x,y)\in K$, we will show that%
\begin{equation}
f(x^{\ast },y^{\ast })+\lambda \sum_{i=1}^{n} r_{cap}(x_{i}^{\ast
})\leq f(x,y)+\lambda \sum_{i=1}^{n}r_{cap}(x_{i}).  \label{ieqpsi}
\end{equation}%
By the assumption that $(x^{\ast },y^{\ast },u^{\ast })$ is an optimal
solution of \eqref{eqn:2}, we have 
\begin{equation}
f(x^{\ast },y^{\ast })+\lambda \sum_{i=1}^{n}\pi (u_{i}^{\ast })\leq
f(x,y)+\lambda \sum_{i=1}^{n}\pi (u_{i})  \label{ieqpi}
\end{equation}%
for any feasible solution $(x,y,u)$ of \eqref{eqn:2}. Let $\ $ 
\begin{equation*}
u_{i}^{x}\in \arg \min \left\{ \pi (\xi )\ :\xi \in \left\{ \frac{|x_{i}|}{M}%
,1\right\} \right\} \subset \arg \min \left\{ \pi (\xi ):\frac{|x_{i}|}{M}%
\leq \xi \leq 1\right\} ,
\end{equation*}%
for all $i=1,\dots ,n$. Then $(x,y,u^{x})$ is a feasible solution of %
\eqref{eqn:1} and 
\begin{equation*}
\pi (u_{i}^{x})=\min \left\{ \pi (\xi ):\frac{|x_{i}|}{M}\leq \xi \leq
1\right\} \ =r_{cap}(x_{i}),\quad \forall i=1,\dots ,n.
\end{equation*}%
Combining (\ref{ieqpi}) in which $u_{i}$ is replaced by $u_{i}^{x}$ and the
last equation we get (\ref{ieqpsi}), which implies that $(x^{\ast },y^{\ast
})$ is an optimal solution of \eqref{eqn:3}.

Conversely, if $(x^{\ast },y^{\ast })$ is a solution of \eqref{eqn:3}, and
let $u_{i}^{\ast }\in \left\{ \frac{|x_{i}^{\ast }|}{M},1\right\} \ $such
that $\pi (u_{i}^{\ast })=r_{cap}(x_{i}^{\ast })$ for $i=1,\dots ,n.$ Then $(x^{\ast },y^{\ast
},u^{\ast })$ is a feasible solution of \eqref{eqn:2} and for an arbitrary
feasible solution $(x,y,u)$ of \eqref{eqn:2}, we have 
\begin{eqnarray*}
f(x,y)+\lambda \sum_{i=1}^{n}\pi(u_i) \geq f(x,y)+\lambda \sum_{i=1}^{n}r_{cap}(x_{i})\\\geq
f(x^{\ast },y^{\ast })+\lambda \sum_{i=1}^{n}r_{cap}(x_{i}^{\ast
})
=f(x^{\ast },y^{\ast })+\lambda \sum_{i=1}^{n}\pi(u^*_i) .
\end{eqnarray*}
Thus, $(x^{\ast },y^{\ast },u^{\ast })$ is an optimal solution of %
\eqref{eqn:2}. The equality $\alpha (\tau )=\beta (\theta )$ is immediately
deduced from the equality $\pi (u_{i}^{\ast })=r_{cap}(x_{i}^{\ast }).$%
\end{proof}

We conclude from the above results that for $\theta =\frac{\tau +\lambda }{%
\lambda M}$ with $\tau >\max \{\lambda ,\tau _{0}\}$, or equivalently $%
\theta > \theta_0 := \max\{\frac{2}{M},\frac{\tau_0+\lambda}{\lambda M}\}$,
the approximate problem \eqref{eqn:3} is equivalent to the origial problem (%
\ref{mainpbl0}). The result justifies the goodness of the Capped-$\ell _{1}$
approximation studied in Section \ref{Sect.compare} above. 

\subsection{A special case: link between the original problem 
\eqref{mainpbl0} and Capped-$\ell _{1}$ approximation problem}

In particular, for a special structure of $K$, we get the following result.

\begin{proposition}
\label{prop:equiv-spe} Suppose that $K = \prod_{i=1}^{n} [-l_i,l_i]\times Y~
(0\le l_i \le +\infty~ \forall i, Y \subset \mathbb{R}^m)$ and $\kappa>0$ is
a constant satisfying 
\begin{equation}  \label{def:L}
|f(x,y)-f(x^{\prime },y)| \le \kappa \|x-x^{\prime }\|_2 \quad \forall
(x,y),(x^{\prime },y) \in K, \|x-x^{\prime }\|_0\le 1.
\end{equation}
Then for $\theta > \frac{\kappa}{\lambda}$, the problems \eqref{mainpbl0}
and \eqref{eqn:3} are equivalent.
\end{proposition}

\begin{proof}
We observe that if $(x,y) \in K$ such that $0< |x_{i_0}| <\frac{1}{\theta}$ for some $i_0$, let $(x',y) \in K$ determined by $x'_i=x_i ~\forall i\ne i_0$ and $x'_{i_0}=0$, then 
\begin{equation*}
f(x,y) + \lambda \Phi(x) > f(x',y) + \lambda \Phi(x'),
\end{equation*}
where $\Phi(x) = \sum_{i=1}^n r_{cap}(x_i)$. Indeed, this inequality follows the facts that 
\[
|f(x,y) - f(x',y)| \le \kappa \|x-x'\| = \kappa|x_{i_0}|
\]
and 
\[
\Phi(x) - \Phi(x') = r_{cap}(x_{i_0}) = \theta |x_{i_0}| > \frac{\kappa}{\lambda} |x_{i_0}|.
\]
For $x \in \mathbb{R}^n$, we define $t^x \in \mathbb{R}^n$ by $t^x_i = 0$ if $|x_i| < \frac{1}{\theta}$ and $t^x_i=x_i$ otherwise. By applying the above observation, for any $(x,y) \in K$, we have
\begin{equation*}
f(x,y) + \lambda \Phi(x) \ge f(t^x,y) + \lambda \Phi(t^x).
\end{equation*}
The equality holds iff $|x_i| \ge \frac{1}{\theta}~\forall i\in \mathrm{supp}(x)$.

Therefore, if $(x^*,y^*)$ is a solution of \eqref{eqn:3}, we have $|x^*_i| \ge \frac{1}{\theta} ~\forall i\in \mathrm{supp}(x^*)$. Then, for any $(x,y) \in K$,
\begin{eqnarray*}
f(x,y) + \lambda \|x\|_0 \ge f(x,y) + \lambda \Phi(x) \ge f(x^*,y^*) + \lambda \Phi(x^*) = f(x^*,y^*) + \lambda \|x^*\|_0.
\end{eqnarray*}
This means that $(x^*,y^*)$ is a solution of \eqref{mainpbl0}. 

Conversely, assume that $(x^*,y^*)$ is a solution of \eqref{mainpbl0}. Then for any $(x,y) \in K$, we have
\begin{eqnarray*}
f(x,y) + \lambda \Phi(x) &\ge & f(t^x,y) + \lambda \Phi(t^x) = f(t^x,y) + \lambda \|t^x\|_0 \\
& \ge & f(x^*,y^*) + \lambda \|x^*\|_0 \ge f(x^*,y^*) + \lambda \Phi(x^*).
\end{eqnarray*}
Thus, $(x^*,y^*)$ is a solution of \eqref{eqn:3}.
\end{proof}

For the problem of feature selection in SVM, we consider the loss function 
\begin{eqnarray*}
f(x,b) = (1-\lambda )\left(\frac{1}{N_A}\|\max\{0,-Ax+eb +e\}\|_1 +\frac{1}{%
N_B} \|\max\{0,Bx-eb +e\}\|_1 \right),
\end{eqnarray*}
(cf. Sect. \ref{num} for definition of notations).

It is easy to prove that for $u\in \mathbb{R}^{n}$, $\iota \in \mathbb{R}$
and $i\in \{1,\dots ,n\}$, we have 
\begin{equation*}
|\max \{0,\langle u,x\rangle +\iota \}-\max \{0,\langle u,x^{\prime }\rangle
+\iota \}|\leq |u|_{i}|x_{i}-x_{i}^{\prime }|,
\end{equation*}%
for all $x,x^{\prime }\in \mathbb{R}^{n}$ such that $x_{j}=x_{j}^{\prime
}~\forall j\neq i$. Therefore, for $\kappa =(1-\lambda
)\max\limits_{i=1,\dots ,n}\left\{ \frac{1}{N_{A}}\sum%
\limits_{k=1}^{N_{A}}|A_{ki}|+\frac{1}{N_{B}}\sum%
\limits_{l=1}^{N_{B}}|B_{li}|\right\} $, we have 
\begin{equation*}
|f(x,b)-f(x^{\prime },b)|\leq \kappa \Vert x-x^{\prime }\Vert ,\quad \forall
b\in \mathbb{R},\forall x,x^{\prime }\in \mathbb{R}^{n}\text{ s.t. }\Vert
x-x^{\prime }\Vert _{0}\leq 1.
\end{equation*}%
By virtue of Proposition \ref{prop:equiv-spe}, in the case of feature
selection in SVM, for $\theta >\theta ^{\ast }:=\frac{\kappa }{\lambda }$,
the problems \eqref{mainpbl0} and \eqref{eqn:3} are equivalent.


\subsection{Extension to other approximations}

\begin{proposition}
\label{propo_equi} i) Suppose that $\sigma $ is a function on $\mathbb{R}$
satisfying 
\begin{equation*}
r_{cap}(t)\leq \sigma (t)\leq s(t)=%
\begin{cases}
0, & \text{if }t=0, \\ 
1, & \text{otherwise},%
\end{cases}%
\end{equation*}%
for some $\theta_{cap} >\theta _{0}$. Then, the problems \eqref{mainpbl0}
and 
\begin{equation}
\inf \{f(x,y)+\lambda \sum_{i=1}^{n}\sigma (x_{i}):(x,y)\in K\}  \label{eta}
\end{equation}%
are equivalent. \newline
ii) In particular, if $\theta_{scad} > a \theta_0$ 
then for all $\tau \geq \lambda $ the approximate problem%
\begin{equation*}
\inf \{f(x,y)+\lambda \sum_{i=1}^{n}r_{scad}(x_{i}):(x,y)\in K\}
\end{equation*}
is equivalent to \eqref{mainpbl0}.
\end{proposition}

\begin{proof}
As discussed before, since $\theta_{cap} > \theta_0$, the problems \eqref{mainpbl0} and \eqref{eqn:3} are equivalent. Moreover, if $(x^*,y^*)$ is a common solution then
\[
f(x^*,y^*) + \lambda \sum_{i=1}^n r_{cap}(x^*_i) = f(x^*,y^*) + \lambda \|x^*\|_0.
\]
Then i) is trivial by the fact that
\[
f(x,y) + \lambda \sum_{i=1}^n r_{cap}(x_i) \le f(x,y) + \lambda \sum_{i=1}^n \sigma(x_i) \le f(x,y) + \lambda \|x\|_0, \quad \forall (x,y).
\]
ii) is  a direct consequence of i) and Propositions \ref{cap} and     (\ref{scad-cap}).
\end{proof}


\section{DCA for solving the problem \eqref{prob:app}}

\label{DCA-prox}

In this section, we will omit the parameter $\theta $ when this doesn't
cause any ambiguity.

Usual sparsity-inducing functions are concave, increasing on $[0,+\infty )$.
Therefore, first we present three variants of DCA for solving the problem %
\eqref{prob:app} when $r$ is concave on $[0,+\infty )$. We also suppose that 
$r$ has the right derivative at $0$, denoted by $r^{\prime }(0)$, so $%
\partial (-r)(0)=\{-r^{\prime }(0)\}$.

First, we consider the approximate problem (\ref{prob:app}).

\subsection{The first DCA scheme for solving the problem \eqref{prob:app}}

We propose the following DC decomposition of $r$: 
\begin{equation}
r(t)=\eta |t|-(\eta |t|-r(t))\quad \forall t\in \mathbb{R},  \label{decompr}
\end{equation}%
where $\eta $ is a positive number such that $\psi (t)=\eta |t|-r(t)$ is
convex. The next result gives a sufficient condition for the existence of
such a $\eta $.

\begin{proposition}
\label{prop:rDC} Suppose that $r$ is a concave function on $[0,+\infty )$
and the (right) derivative at 0, $r^{\prime }(0)$, is well-defined. Let $%
\eta \geq r^{\prime }(0)$. Then $\psi (t)=\eta |t|-r(|t|)$ is a convex
function on $\mathbb{R}$.
\end{proposition}

\begin{proof}
Since $r$ is concave on 
$[0,+\infty)$, the function $\eta |t| - r(t)$  is convex on $(0,+\infty)$ and on $(-\infty,0)$. Hence it  suffices to prove that for any $t_1 > 0,t_2 < 0$ and $\alpha,\beta \in (0,1)$ such that $\alpha + \beta = 1$, we have
\begin{equation}\label{eqn:inq}
\psi\left(\alpha t_1+\beta t_2\right) \le \alpha \psi(t_1) + \beta \psi(t_2).
\end{equation}
Without loss of generality, we assume that $\alpha|t_1| \ge \beta |t_2|$. Then \eqref{eqn:inq} is equivalent to
\[
\eta (\alpha|t_1|-\beta|t_2|) - 2r\left(\alpha|t_1|-\beta|t_2|\right) \le \eta (\alpha|t_1| + \beta|t_2|) - \alpha r(|t_1|) - \beta r(|t_2|)
\]
which can be equivalently written as
\begin{equation}\label{eqn:inqb}
 \alpha r(|t_1|) + \beta r(|t_2|) - r\left(t_0\right) \le 2 \eta \beta |t_2|,
\end{equation}
where $t_0 = \alpha|t_1|-\beta|t_2| \ge 0$. Let $\mu \in \mathbb{R}$ such that $-\mu \in \partial (-r(t_0))$. Since $r$ is concave on $[0,+\infty)$, we have
\[
\alpha r(|t_1|) + \beta r(|t_2|) - r\left(t_0\right) \le  r\left(\alpha|t_1|+\beta|t_2|\right) - r\left(t_0\right) \le 2 \mu \beta |t_2|.
\]
Hence (\ref{eqn:inqb})  holds when $\mu \le \eta$. By the concavity of $r$, we have
\begin{eqnarray*}
r\left(\frac{t_0}{2}\right) \le  r(0) + r'(0) \frac{t_0}{2}, \quad \text{and} \quad r\left(\frac{t_0}{2}\right) \le  r(t_0) - \mu \frac{t_0}{2},
\end{eqnarray*}
therefore
\[(z - r'(0)) t_0 \le r(0) + r(t_0) - 2 r\left(\frac{t_0}{2}\right) \le 0.
\]
This and the condition  $ r'(0) \le \eta$ imply that $\mu \le r'(0) \le \eta$. The proof is then complete.
\end{proof}

With $\eta \geq r^{\prime }(0)$, a DC formulation of the problem %
\eqref{prob:app} is given by 
\begin{equation}
\min_{x,y}\{F_{r}(x,y):=G_{1}(x,y)-H_{1}(x,y)\},  \label{eqn:DC1}
\end{equation}%
where 
\begin{equation*}
G_{1}(x,y)=\chi _{K}(x,y)+g(x,y)+\lambda \eta \Vert x\Vert _{1},~
H_{1}(x,y)=h(x,y)+\lambda \sum_{i=1}^{n}\left( \eta |x_{i}|-r(x_{i})\right) ,
\end{equation*}%
and $g,h$ are DC components of $f$.

By the definition $\psi (t)=\eta |t|-r(t)~\forall t\in \mathbb{R}$, we have 
\begin{equation}  \label{subgrad:DCA1}
\partial \psi (t)=\eta +\partial (-r)(t) \text{ if }t>0, \: -\eta -\partial
(-r)(-t) \text{ if }t<0, \: \left[ -\eta +r^{\prime }(0),\eta -r^{\prime }(0)%
\right] \text{ if }t=0.
\end{equation}%
Following the generic DCA scheme described in Section \ref{dca}, DCA applied
on \eqref{eqn:DC1} is given by Algorithm \ref{DCA1} below. 
\begin{algorithm}[H]
\caption{DCA for solving \eqref{prob:app} (DCA1)} 
\label{DCA1}
\begin{algorithmic}
\STATE Initialize $(x^0,y^0) \in K$, $k \leftarrow 0$
\REPEAT
\STATE 1. Compute $(\overline{x}^k,\overline{y}^k) \in \partial h(x^k,y^k)$ and $\overline{z}^k_i \in \lambda \partial \psi(x^k_i)~\forall i=1,\dots,n$ via \eqref{subgrad:DCA1}.
\STATE 2. Compute
\[
(x^{k+1},y^{k+1}) \in \arg \min_{(x,y) \in K}  \left\{g(x,y) - \langle \overline{x}^k,x \rangle - \langle \overline{y}^k,y\rangle + \lambda \eta \|x\|_1 - \langle \overline{z}^k, x \rangle \right\}
\]
\STATE 3. $k \leftarrow k+1$.
\UNTIL{Stopping criterion}
\end{algorithmic}
\end{algorithm}

Instances of Algorithm~1 can be found in our previous works \cite%
{LeThietal2008a,LeThietal2009,Ong12} using exponential concave, SCAD or
Capped$-\ell _{1}$ approximations (see Table~\ref{tab:DCA1}). Note that for
usual sparse inducing functions given in Table~\ref{tab:DCA1}, this DC
decomposition is nothing but that given in Table~\ref{tab:app-form}, i.e. $%
\varphi (t)=\eta |t|$.

\begin{table}[tbp]
\caption{Choice of $\protect\eta$ and expression of $\overline{z}_i^k \in 
\protect\lambda \partial \protect\psi(x_i^k)$ in Algorithm \protect\ref{DCA1}
and related works.}
\label{tab:DCA1}\centering
\begin{tabular}{lllcl}
\hline
$r$ & $\eta$ & $\overline{z}_i^k \in \lambda \partial \psi(x_i^k)$ & Related
works & Context \\ \hline
\multirow{2}{*}{$r_{exp}$} & \multirow{2}{*}{$\theta$} & \multirow{2}{*}{$%
\mathrm{sign}(x_i^k) \lambda \theta \left(1-e^{-\theta |x_i^k|}\right)$} & 
\cite{LeThietal2008a} & Feature selection in SVMs \\ 
&  &  & \cite{Ong12} & Learning sparse classifiers \vspace{0.1cm} \\ 
$r_{\ell_p^+}$ & $\frac{\epsilon^{1/\theta - 1}}{\theta}$ & $\mathrm{sign}%
(x_i^k) \frac{\lambda}{\theta} \left[\epsilon^{1/\theta-1} -
(|x_i^k|+\epsilon)^{1/\theta-1}\right]$ &  & \vspace{0.1cm} \\ 
$r_{\ell_p^-}$ & $-p\theta$ & $-\mathrm{sign}(x_i^k) \lambda p \theta \left[%
1 - (1+\theta |x^k_i|)^{p-1}\right]$ &  & \vspace{0.1cm} \\ 
$r_{log}$ & $\frac{\theta}{\log(1+\theta)}$ & $\mathrm{sign}(x_i^k) \frac{%
\lambda \theta^2 |x^k_i|}{\log(1+\theta) (1+\theta |x^k_i|)}$ &  &  \\ 
$r_{scad}$ & $\frac{2\theta}{a+1}$ & $%
\begin{cases}
0 & |x_i^k| \le \frac{1}{\theta} \\ 
\mathrm{sign}(x_i^k) \frac{2\lambda \theta (\theta|x_i^k|-1)}{a^2-1} & \frac{%
1}{\theta} < |x_i^k| < \frac{a}{\theta} \\ 
\mathrm{sign}(x_i^k) \frac{2\lambda \theta}{a+1} & \text{otherwise}%
\end{cases}%
$ & \cite{LeThietal2009} & Feature selection in SVMs \\ 
$r_{cap}$ & $\theta$ & $%
\begin{cases}
0 & |x_i^k| \le \frac{1}{\theta} \\ 
\mathrm{sign}(x_i^k) \lambda \theta & \text{otherwise}%
\end{cases}%
$ & \cite{Ong12} & Learning sparse classifiers \\ \hline
\end{tabular}
\end{table}

Now we consider the approximate problem (\ref{prob:app-z}) and introduce a
DCA scheme that includes all standard algorithms of reweighted-$\ell _{1}$%
-type for sparse optimization problem (\ref{prob:l0}).

\subsection{DCA2 - Relation with reweighted-$\ell _{1}$ procedure}

\label{sec.DCA2}

The problem \eqref{prob:app-z} can be written as a DC program as follows 
\begin{equation}
\min_{x,y,z}\{\overline{F}_{r}(x,y,z):=G_{2}(x,y,z)-H_{2}(x,y,z)\},
\label{eqn:DC2}
\end{equation}%
where 
\begin{equation*}
G_{2}(x,y,z)=\chi _{\Omega _{1}}(x,y,z)+g(x,y), ~
H_{2}(x,y,z)=h(x,y)+\lambda \sum_{i=1}^{n}(-r)(z_{i}),
\end{equation*}%
and $g,h$ are DC components of $f$ as stated in \eqref{eqn:fDC}.

Assume that $(x^{k},y^{k},z^{k})\in \Omega _{1}$ is the current solution at
iteration $k$. DCA applied to DC program \eqref{eqn:DC2} updates $%
(x^{k+1},y^{k+1},z^{k+1})\in \Omega _{1}$ via two steps:

\begin{enumerate}
\item[-] Step 1: compute $(\overline{x}^{k},\overline{y}^{k})\in \partial
h(x^{k},y^{k})$, and $\overline{z}_{i}^{k}\in \lambda \partial
(-r)(z_{i}^{k})~~\forall i=1,\dots,n$.

\item[-] Step 2: compute 
\begin{eqnarray*}
(x^{k+1},y^{k+1},z^{k+1}) &\in &\arg \min\left\{ G_{2}(x,y,z)-\langle 
\overline{x}^{k},x\rangle -\langle \overline{y}^{k},y\rangle -\langle 
\overline{z}^{k},z\rangle \right\} \\
&=&\mathop{\arg \min}_{(x,y,z)\in \Omega _{1}}\left\{ g(x,y)-\langle 
\overline{x}^{k},x\rangle -\langle \overline{y}^{k},y\rangle +\langle -%
\overline{z}^{k},z\rangle \right\} .
\end{eqnarray*}
\end{enumerate}

Since $r$ is increasing, we have $-\overline{z}^{k}\geq 0$. Thus, updating $%
(x^{k+1},y^{k+1},z^{k+1})$ can be done as follows 
\begin{equation*}
\begin{cases}
(x^{k+1},y^{k+1})\in \arg \min_{(x,y)\in K}\left\{ g(x,y)-\langle \overline{x%
}^{k},x\rangle -\langle \overline{y}^{k},y\rangle +\langle -\overline{z}%
^{k},|x|\rangle \right\} \\ 
z_{i}^{k+1}=|x_{i}^{k+1}|~~\forall i.%
\end{cases}%
\end{equation*}%
DCA for solving the problem \eqref{prob:app-z} can be described as in
Algorithm \ref{DCA2} below.

\begin{algorithm}[H]
\caption{DCA for solving (\ref{prob:app-z}) (DCA2)}
\label{DCA2}
\begin{algorithmic}
\STATE Initialize $(x^0,y^0,z^0) \in \Omega_1$, $k \leftarrow 0$
\REPEAT
\STATE 1. Compute $(\overline{x}^k,\overline{y}^k) \in \partial h(x^k,y^k)$, $\overline{z}^k_i \in -\lambda \partial(-r)(z^k_i)~~\forall i=1,\dots,n$.
\STATE 2. Compute
\begin{eqnarray*}
&&(x^{k+1},y^{k+1}) \in \mathop{\arg \min}_{(x,y) \in K}  \left\{g(x,y) - \langle \overline{x}^k,x \rangle - \langle \overline{y}^k,y \rangle + \langle \overline{z}^k,|x| \rangle\right\} \\
&&z^{k+1}_i = |x^{k+1}_i|~\forall i=1,\dots,n.
\end{eqnarray*}
\STATE 3. $k \leftarrow k+1$.
\UNTIL{Stopping criterion}
\end{algorithmic}
\end{algorithm}

If the function $f$ in (\ref{prob:l0}) is convex, we can chose DC components
of $f$ as $g=f$ and $h=0$. Then $(\overline{x}^{k},\overline{y}^{k})=0
~\forall k$. 
In this case, the step 2 in Algorithm \ref{DCA2} becomes 
\begin{equation}
(x^{k+1},y^{k+1})\in \arg \min_{(x,y)\in K}\left\{ f(x,y)+\sum_{i=1}^{n} 
\overline{z}_{i}^{k}|x_{i}|\right\} .  \label{eqn:weighted_l1}
\end{equation}%
We see that the problem \eqref{eqn:weighted_l1} has the form of a $\ell _{1}$%
-regularization problem but with different weights on components of $|x_{i}|$%
. So Algorithm \ref{DCA2} iteratively solves the weighted--$\ell _{1}$
problem \eqref{eqn:weighted_l1} with an update of the weights $\overline{z}%
_{i}^{k}$ at each iteration $k$. The expression of weights $\overline{z}%
_{i}^{k}$ according to approximation functions are given in Table \ref%
{tab:l1}.

The update rule \eqref{eqn:weighted_l1} covers standard algorithms of
reweighted--$\ell _{1}$--type for sparse optimization problem (\ref{prob:l0}%
) (see Table \ref{tab:l1}). Some algorithms such as the two--stage $\ell
_{1} $ (\cite{Zhang09}) and the adaptive Lasso (\cite{ZOU06}) only run in a
few iterations (typically two iterations) and their reasonings bear a
heuristic character. The reweighted--$\ell _{1}$ algorithm proposed in \cite%
{Candes08} lacks of theoretical justification for the convergence. 

\begin{table}[tbp]
\caption{Expression of $\overline{z}^k_i $ in Algorithm \protect\ref{DCA2}
and relation with reweighted-$\ell_1$ algorithms. }
\label{tab:l1}\centering
\begin{tabular}{llp{3cm}ll}
\hline
Function $r$ & expression of $\overline{z}^k_i$ & Related works & Context & 
\\ \hline\hline
$r_{exp}$ & $\lambda\theta e^{-\theta z^k_i}$ & SLA (\cite%
{Bradley-Mangasarian}) & Feature selection in SVMs &  \\ \hline
$r_{\ell_p^+}$ & $\dfrac{\lambda}{\theta(z^k_i + \epsilon)^{1-1/\theta}}$ & %
\multirow{2}{*}{Adaptive Lasso (\cite{ZOU06})} & \multirow{8}{*}{Linear
regression} \vspace{0.1cm} &  \\ 
\cmidrule{1-2} $r_{\ell_p^-}$ & $-\lambda p \theta (1+\theta z^k_i)^{p-1}$ & 
&  &  \\ 
\cmidrule{1-3} $r_{scad}$ & $%
\begin{cases}
\frac{2 \lambda \theta}{a+1} & \text{if } z^k_i \le \frac{1}{\theta} \\ 
0 & \text{if } z^k_i \ge \frac{a}{\theta} \\ 
\frac{\lambda \theta(a-\theta z^k_i)}{a^2-1} & \text{otherwise}%
\end{cases}%
$ & LLA (Local Linear Approximation) (\cite{Zou08}) &  &  \\ 
\cmidrule{1-3} $r_{cap}$ & $%
\begin{cases}
\lambda\theta & \text{if } z^k_i \le 1/\theta \\ 
0 & \text{otherwise}%
\end{cases}%
$ & Two-stage $\ell_1$ (\cite{Zhang09}) &  &  \\ \hline
$r_{log}$ & $\dfrac{\lambda \theta}{\log(1+\theta) (1 + \theta z^k_i)} $ & 
Adaptive Lasso (\cite{ZOU06}); Reweighted $\ell_1$ (\cite{Candes08}) & 
Sparse signal reconstruction &  \\ \hline
\end{tabular}
\end{table}

Next, we introduce a slight perturbation of the formulation (\ref{prob:app})
and develop the third DCA scheme that includes existing algorithms of
reweighted--$\ell _{2}$--type for sparse optimization problem (\ref{prob:l0}%
).

\subsection{DCA3 - Relation with reweighted-$\ell _{2}$ procedure}

\label{sec.DCA3}

To avoid the singularity at $0$ of the function $r(t^{1/2}),t\geq 0$, we add 
$\epsilon >0$ and consider the perturbation problem of \eqref{prob:app}
which is defined by

\begin{equation}
\begin{cases}
\min_{x,y} & \tilde{F}_{r}(x,y):=f(x,y)+\lambda
\sum_{i=1}^{n}r((|x_{i}|^{2}+\epsilon )^{1/2}) \\ 
s.t. & (x,y)\in K,%
\end{cases}%
\quad \epsilon >0.  \label{prob:app-modified}
\end{equation}%
Clearly \eqref{prob:app-modified} becomes \eqref{prob:app} when $\epsilon =
0 $. The problem \eqref{prob:app-modified} is equivalent to 
\begin{equation}  \label{pb3}
\min_{(x,y,z)\in \Omega _{2}}\hat{F}_{r}(x,y,z):=f(x,y)+\lambda
\sum_{i=1}^{n}r((z_{i}+\epsilon )^{1/2}),
\end{equation}%
where $\Omega _{2}=\{(x,y,z):(x,y)\in K;\quad |x_{i}|^{2}\leq z_{i}~\forall
i\}$. The last problem is a DC program of the form 
\begin{equation}
\min_{x,y,z}\{\hat{F}_{r}(x,y,z):=G_{3}(x,y,z)-H_{3}(x,y,z)\},
\label{eqn:DC3}
\end{equation}%
where 
\begin{equation*}
G_{3}(x,y,z)=\chi _{\Omega _{2}}(x,y,z)+g(x,y),~ H_{3}(x,y,z)=h(x,y)+\lambda
\sum_{i=1}^{n}(-r)((z_{i}+\epsilon )^{1/2}),
\end{equation*}%
and $g,h$ are DC components of $f$ as stated in \eqref{eqn:fDC}. Note that,
since the functions $r$ and $(t+\epsilon )^{1/2}$ are concave, increasing on 
$[0,+\infty )$, $(-r)((t+\epsilon )^{1/2})$ is a convex function on $%
[0,+\infty )$.

Let $(x^{k},y^{k},z^{k})\in \Omega _{2}$ be the current solution at
iteration $k$. DCA applied to DC program \eqref{eqn:DC3} updates $%
(x^{k+1},y^{k+1},z^{k+1})\in \Omega _{2}$ via two steps:

\begin{enumerate}
\item[-] Step 1: compute $(\overline{x}^{k},\overline{y}^{k})\in \partial
h(x^{k},y^{k})$, and $\overline{z}_{i}^{k}\in \frac{\lambda }{%
2(z_{i}^{k}+\epsilon )^{1/2}}\partial (-r)((z_{i}^{k}+\epsilon
)^{1/2})~~\forall i=1,\dots,n$.

\item[-] Step 2: compute 
\begin{eqnarray*}
(x^{k+1},y^{k+1},z^{k+1}) &\in &\arg \min\left\{ G_{3}(x,y,z)-\langle 
\overline{x}^{k},x\rangle -\langle \overline{y}^{k},y\rangle -\langle 
\overline{z}^{k},z\rangle \right\} \\
&=&\mathop{\arg \min}_{(x,y,z)\in \Omega _{2}}\left\{ g(x,y)-\langle 
\overline{x}^{k},x\rangle -\langle \overline{y}^{k},y\rangle +\langle -%
\overline{z}^{k},z\rangle \right\}
\end{eqnarray*}
\end{enumerate}

Since $r$ is increasing, we have $-\overline{z}^{k}\geq 0$. Thus, updating $%
(x^{k+1},y^{k+1},z^{k+1})$ can be done as follows 
\begin{equation*}
\begin{cases}
(x^{k+1},y^{k+1})\in \mathop{\arg \min}_{(x,y)\in K}\left\{ g(x,y)-\langle 
\overline{x}^{k},x\rangle -\langle \overline{y}^{k},y\rangle
+\sum_{i=1}^{n}(-\overline{z}_{i}^{k})x_{i}^{2}\rangle \right\} \\ 
z_{i}^{k+1}=|x_{i}^{k+1}|^{2}~~\forall i=1,\dots ,n.%
\end{cases}%
\end{equation*}%
DCA for solving the problem \eqref{pb3} can be described as in Algorithm \ref%
{DCA3} below.

\begin{algorithm} 
\caption{DCA for solving (\ref{pb3}) (DCA3)}
\label{DCA3}
\begin{algorithmic}
\STATE Initialize $(x^0,y^0,z^0) \in \Omega_2$, $k \leftarrow 0$
\REPEAT
\STATE 1. Compute $(\overline{x}^k,\overline{y}^k) \in \partial h(x^k,y^k)$, $\overline{z}^k_i \in  \frac{-\lambda}{2 (z^k_i+\epsilon)^{1/2}}\partial(-r)(z^k_i+\epsilon)^{1/2})~~\forall i=1,\dots,n$.
\STATE 2. Compute
\begin{eqnarray*}
&&(x^{k+1},y^{k+1}) \in \mathop{\arg \min}_{(x,y) \in K}  \left\{g(x,y) - \langle \overline{x}^k,x \rangle - \langle \overline{y}^k,y \rangle + \sum_{i=1}^n \overline{z}^k_i x^2_i \right\}, \\
&&z_{i}^{k+1}=|x_{i}^{k+1}|^{2}~~\forall i=1,\dots ,n.
\end{eqnarray*}
\STATE 3. $k \leftarrow k+1$.
\UNTIL{Stopping criterion}
\end{algorithmic}
\end{algorithm}

If the function $f$ in (\ref{prob:l0}) is convex, then, as before, we can
chose DC components of $f$ as $g=f$ and $h=0$. Hence, in the step 1 of
Algorithm \ref{DCA3}, we have $(\overline{x}^{k},\overline{y}^{k})=0~\forall
k$. 
In this case, the step 2 in Algorithm \ref{DCA3} becomes 
\begin{equation}
(x^{k+1},y^{k+1})\in \arg \min_{(x,y)\in K}\left\{ f(x,y)+\sum_{i=1}^{n}%
\overline{z}_{i}^{k}x_{i}^{2}\right\} .  \label{eqn:weighted_l2}
\end{equation}%
Thus, each iteration of Algorithm \ref{DCA3} solves a weighted-$\ell _{2}$
optimization problem. The expression of weights $\overline{z}_{i}^{k}$
according to approximation functions are given in Table \ref{tab:l2}.

If $\epsilon =0$ then the update rule \eqref{eqn:weighted_l2} encompasses
standard algorithms of reweighted-$\ell _{2}$ type for finding sparse
solution (see Table \ref{tab:l2}). However, when $\epsilon =0$ the (right)
derivative at 0 of $r(t^{1/2})$\ is not well-defined, that is why we take $%
\epsilon >0$ in our algorithm. Note also that, in LQA and FOCUSS, if at an
iteration $k$ one has $x_{i}^{k}=0$\ then $x_{i}^{l}=0$ for all $l\geq k,$
by the way these algorithms may converge prematurely to bad solutions.

\begin{table}[tbp]
\caption{Expression of $\overline{z}^k_i$'s in Algorithm \protect\ref{DCA3}
and relation with reweighted-$\ell_2$ algorithms.}
\label{tab:l2}\setlength{\tabcolsep}{0.2cm} \centering
\begin{tabular}{llp{4cm}ll}
\hline
Function $r$ & weight $\overline{z}^k_i$($t^k_i = (z^k_i+\epsilon)^{1/2}$) & 
Related works & Context &  \\ \hline\hline
$r_{exp}$ & $\dfrac{\lambda \theta}{2} \dfrac{e^{-\theta t^k_i}}{t^k_i}$ & 
&  & \vspace{0.1cm} \\ \hline
$r_{\ell_p^+}$ & $\dfrac{\lambda}{2\theta (t^k_i)^{2 - \frac{1}{\theta}}} $
& \multirow{3}{*}{FOCUSS (\cite{Rao97,Rao99,Rao03}); } & %
\multirow{3}{*}{Sparse signal} &  \\ 
\cmidrule{1-2} $r_{\ell_p^-}$ & $\dfrac{-\lambda p \theta^p}{2 t^k_i (\frac{1%
}{\theta} + t^k_i)^{1-p}}$ & \multirow{3}{*}{IRLS (\cite{Chartrand08})} & %
\multirow{3}{*}{reconstruction} &  \\ 
\cmidrule{1-2} $r_{log}$ & $\dfrac{\lambda}{2 \log(1+\theta)} \dfrac{1}{%
t^k_i(\frac{1}{\theta}+ t^k_i)}$ &  & \vspace{0.1cm} &  \\ \hline
$r_{cap}$ & $%
\begin{cases}
\frac{\lambda\theta}{2t^k_i} & \text{if } |t^k_i| \le \frac{1}{\theta} \\ 
0 & \text{otherwise}%
\end{cases}%
$ & \vspace{0.1cm} &  &  \\ \hline
$r_{scad}$ & $%
\begin{cases}
\frac{\lambda \theta}{(a+1) t^k_i} & \text{if } t^k_i \le \frac{1}{\theta}
\\ 
0 & \text{if } t^k_i \ge \frac{a}{\theta} \\ 
\frac{\lambda \theta (-\theta t^k_i + a)}{(a^2-1)t^k_i} & \text{otherwise}%
\end{cases}%
$ & LQA (\cite{FAN01,Zou08}) & Linear regression &  \\ \hline
\end{tabular}
\end{table}

\subsection{Discussion on the three DCA based algorithms \protect\ref{DCA1}, \protect\ref{DCA2}
and \protect\ref{DCA3}}

Algorithm \ref{DCA1} seems to be the most interesting in the sense that it
addresses directly the problem \eqref{prob:app} and doesn't need the
additional variable $z$, then the subproblem has less constraints than that in
Algorithms \ref{DCA2} and \ref{DCA3}. Moreover, the DC decomposition (\ref%
{decompr}) is more suitable since it results, in several cases, in a DC
polyhedral program where both DC components are polyhedral convex (for
instance, in feature selection in SVM with the approximations $r_{scad},$ $%
r_{cap}$) for which Algorithm \ref{DCA1}  enjoys interesting convergence properties. 

Algorithms \ref{DCA2} and \ref{DCA3} are based on two
different formulations of the problem \eqref{prob:app}. In \eqref{prob:app-z}%
, we have linear constraints $|x|_{i}\leq z_{i},~i=1,\dots ,n$ that lead to
the subproblem of weighted--$\ell _{1}$ type. Whereas, in %
\eqref{prob:app-modified}, quadratic constraints $|x|_{i}^{2}\leq
z_{i},~i=1,\dots ,n$ result to the subproblem of weighted--$\ell _{2}$ type.
With second order terms in subproblems, Algorithm \ref{DCA3} is, in general,
more expensive than Algorithms \ref{DCA1} and \ref{DCA2}. We also see that
Algorithms \ref{DCA1} and \ref{DCA2} possess nicer convergence properties
than Algorithm \ref{DCA3}. Both Algorithms \ref{DCA1} and \ref{DCA2} have
finite convergence when the corresponding DC programs are polyhedral DC.
While \eqref{prob:app-modified} can't be a polyhedral DC program because the
set $\Omega _{2}$ and the functions $r((t+\epsilon )^{1/2})$ are not
polyhedral convex.

To compare the sparsity of solutions given by the algorithms, we consider
the subproblems in Algorithms \ref{DCA1}, \ref{DCA2}, and \ref{DCA3} which
have the form 
\begin{equation*}
\min_{(x,y)\in K}\left\{ g(x,y)-\langle \overline{x}^{k},x\rangle -\langle 
\overline{y}^{k},y\rangle +\lambda \sum_{i=1}^{n}\nu
(x_{i},x_{i}^{k})\right\}
\end{equation*}%
where $(\overline{x}^{k},\overline{y}^{k})\in \partial h(x^{k},y^{k})$, 
\begin{equation*}
\nu (x_{i},x_{i}^{k})=%
\begin{cases}
\nu _{1}(x_{i},x_{i}^{k})=\eta |x_{i}|-\mathrm{sign}(x_{i}^{k})(\eta -%
\overline{z}_{i}^{k})x_{i}+C_{i}^{k} & \text{for Algorithm \ref{DCA1}} \\ 
\nu _{2}(x_{i},x_{i}^{k})=\overline{z}_{i}^{k}|x_{i}|+C_{i}^{k} & \text{for
Algorithm \ref{DCA2}} \\ 
\nu _{3}(x_{i},x_{i}^{k})=\frac{\overline{z}_{i}^{k}}{2|x_{i}^{k}|}%
|x_{i}|^{2}+\frac{1}{2}\overline{z}_{i}^{k}|x_{i}^{k}|+C_{i}^{k} & \text{for
Algorithm \ref{DCA3}},%
\end{cases}%
\end{equation*}%
with $\overline{z}_{i}^{k}\in -\partial (-r)(|x_{i}^{k}|)$, $%
C_{i}^{k}=r(x_{i}^{k})-\overline{z}_{i}^{k}|x_{i}^{k}|$ and $\eta =r^{\prime
}(0)$.

All three functions $\nu _{1}$, $\nu _{2}$ and $\nu _{3}$ attain minimum at $%
0$ and encourage solutions to be zero. Denote by $\nu _{-}^{\prime }(t)$ and 
$\nu _{+}^{\prime }(t)$ the left and right derivative at $t$ of $\nu $
respectively. We have 
\begin{eqnarray*}
&&\nu _{1,-}^{\prime }(0,x_{i}^{k})=-2\eta +\overline{z}_{i}^{k},\quad \nu
_{2,-}^{\prime }(0,x_{i}^{k})=-\overline{z}_{i}^{k},\quad \nu _{3,-}^{\prime
}(0,x_{i}^{k})=0, \\
&&\nu _{1,+}^{\prime }(0,x_{i}^{k})=\overline{z}_{i}^{k},\quad \nu
_{2,+}^{\prime }(0,x_{i}^{k})=\overline{z}_{i}^{k},\quad \nu _{3,+}^{\prime
}(0,x_{i}^{k})=0.
\end{eqnarray*}%
We also have $\eta \geq \overline{z}_{i}^{k}$ by the concavity of $r$ on $%
[0,+\infty )$. Observe that if the range $[\nu _{-}^{\prime }(0),\nu
_{+}^{\prime }(0)]$ is large, it encourages more sparsity. Intuitively, the
values $\nu _{-}^{\prime }(0)$ and $\nu _{+}^{\prime }(0)$ reflect the slope
of $\nu $ at $0$, and if the slope is hight, it forces solution to be zero.
Here we have $[\nu _{3,-}^{\prime }(0,x_{i}^{k}),\nu _{3,+}^{\prime
}(0,x_{i}^{k})]\subset \lbrack \nu _{2,-}^{\prime }(0,x_{i}^{k}),\nu
_{2,+}^{\prime }(0,x_{i}^{k})]\subset \lbrack \nu _{1,-}^{\prime
}(0,x_{i}^{k}),\nu _{1,+}^{\prime }(0,x_{i}^{k})]$. Thus, we expect that
Algorithm \ref{DCA1} gives sparser solution than Algorithm \ref{DCA2}, and
Algorithm \ref{DCA2} gives sparser solution than Algorithm \ref{DCA3}.

\begin{figure}[tbp]
\begin{center}
\includegraphics[scale=0.5]{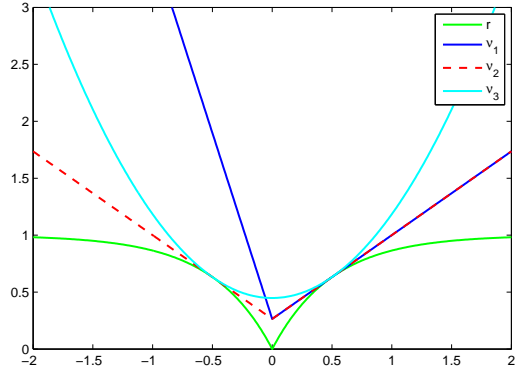}
\end{center}
\caption{Graphs of functions: $r=1-e^{-2|x|}$, $\protect\nu _{1}$, $\protect%
\nu _{2}$ and $\protect\nu _{3}$ with $x^{k}=0.5$.}
\label{Fig.2}
\end{figure}

\subsection{DCA4: DCA applied on \eqref{prob:app} with the new DC
approximation}

\label{new_approximation}

We have proposed three DCA schemes for solving (\ref{prob:app}) or its
equivalent form (\ref{prob:app-z}) when $r$ is a concave function on $%
[0,+\infty )$. Consider now the general case where $r$ is a DC function
satisfying Assumption 1. Hence the problem \eqref{prob:app} can be expressed
as a DC program \eqref{eqn:DC} for which DCA is applicable. Each iteration
of DCA applied on \eqref{eqn:DC} consists of computing

- Compute $(\overline{x}^k,\overline{y}^k) \in \partial h(x^k,y^k)$ and $%
\overline{z}^k_i \in \lambda \partial \psi(x^k_i)~\forall i=1,\dots,n$.

- Compute $(x^{k+1},y^{k+1})$ as a solution of the following convex program 
\begin{equation}
\min_{(x,y)\in K}\left\{ g(x,y)-\langle \overline{x}^{k},x\rangle -\langle 
\overline{y}^{k},y\rangle +\lambda \sum_{i=1}^{n}\varphi (x_{i})-\langle 
\overline{z}^{k},x\rangle \right\} .
\end{equation}

The new approximation function $r_{PiL}$ is a DC function but not concave on 
$[0,+\infty ).$ Hence we apply DCA4 for solving the problem \eqref{prob:app}
with $r=r_{PiL}$ 
\begin{equation}
r_{PiL}=\min \left\{ 1,\max \left\{ 0,\frac{\theta |t|-1}{a-1}\right\}
\right\} =%
\begin{cases}
0 & \text{if }|t|\leq \frac{1}{\theta }, \\ 
\frac{\theta |t|-1}{a-1} & \text{if }\frac{1}{\theta }<|t|<\frac{a}{\theta },
\\ 
1 & \text{otherwise},%
\end{cases}%
\quad a>1.
\end{equation}%
DC components of $r_{PiL}$ are given by 
\begin{equation}
\varphi _{PiL}(t):=\frac{\theta }{a-1}\max \left\{ \frac{1}{\theta }%
,|t|\right\} ,~\psi _{PiL}(t):=\frac{\theta }{a-1}\max \left\{ \frac{a}{%
\theta },|t|\right\} -1\quad \forall t\in \mathbb{R},
\end{equation}%
that are polyhedral convex functions. Then, the problem \eqref{prob:app} can
be expressed in form of a DC program as follows%
\begin{equation}
\min_{x,y}\{F_{r_{PiL}}(x,y):=G_{4}(x,y)-H_{4}(x,y)\},  \label{eqn:DC4}
\end{equation}%
where 
\begin{equation*}
G_{4}(x,y)=\chi _{K}(x,y)+g(x,y)+\lambda \sum_{i=1}^{n}\varphi
_{PiL}(x_{i}),~H_{4}(x,y)=h(x,y)+\lambda \sum_{i=1}^{n}\psi _{PiL}(x_{i}),
\end{equation*}%
and $g,h$ are DC components of $f$ as stated in \eqref{eqn:fDC}.

At each iteration $k$, DCA applied to \eqref{eqn:DC4} updates $%
(x^{k+1},y^{k+1})$ from $(x^k,y^k)$ via two steps:

- Compute $(\overline{x}^k,\overline{y}^k) \in \partial h(x^k,y^k)$ and $%
\overline{z}^k_i \in \lambda \partial \psi_{PiL}(x^k_i)~\forall i=1,\dots,n$.

- Compute $(x^{k+1},y^{k+1})$ as a solution of the following convex program 
\begin{equation}
\min_{(x,y)\in K}\left\{ g(x,y)-\langle \overline{x}^{k},x\rangle -\langle 
\overline{y}^{k},y\rangle +\frac{\lambda \theta }{a-1}\sum_{i=1}^{n}\max
\left\{ \frac{1}{\theta },|x_{i}|\right\} -\langle \overline{z}^{k},x\rangle
\right\} .  \label{eqn:Pk}
\end{equation}%
Calculation of $\overline{z}_{i}^{k}~(i=1,\dots ,n)$ is given by 
\begin{equation}
\overline{z}_{i}^{k}=%
\begin{cases}
\frac{\lambda \theta }{a-1} & \text{if }x_{i}^{k}>\frac{a}{\theta } \\ 
\frac{-\lambda \theta }{a-1} & \text{if }x_{i}^{k}<\frac{-a}{\theta } \\ 
0 & \text{otherwise}.%
\end{cases}
\label{eqn:subgrad}
\end{equation}%
Furthermore, \eqref{eqn:Pk} is equivalent to 
\begin{equation}
\min_{(x,y,t)\in \Omega _{3}}\left\{ g(x,y)-\langle \overline{x}%
^{k},x\rangle -\langle \overline{y}^{k},y\rangle +\frac{\lambda \theta }{a-1}%
\sum_{i=1}^{n}t_{i}-\langle \overline{z}^{k},x\rangle \right\} ,
\label{eqn:Pk-eqv}
\end{equation}%
where $\Omega _{3}=\left\{ (x,y,t):(x,y)\in K,\frac{1}{\theta }\leq
t_{i},x_{i}\leq t_{i},-x_{i}\leq t_{i}~\forall i=1,\dots ,n\right\} $.

\begin{algorithm}[H]
\caption{DCA   applied to \eqref{eqn:DC4} (DCA4)}
\label{DCA4}
\begin{algorithmic}
\STATE Initialize $(x^0,y^0) \in K$, $k \leftarrow 0$
\REPEAT
\STATE 1. Compute $(\overline{x}^k,\overline{y}^k) \in \partial h(x^k,y^k)$ and $\overline{z}^k_i \in \lambda \partial \psi_{PiL}(x^k_i)~\forall i=1,\dots,n$ via \eqref{eqn:subgrad}.
\STATE 2. Solve the convex problem \eqref{eqn:Pk-eqv} to obtain $(x^{k+1},y^{k+1})$.
\STATE 3. $k \leftarrow k+1$.
\UNTIL{Stopping criterion.}
\end{algorithmic}
\end{algorithm}


\subsection{Updating $\protect\theta $ procedure}

\label{sect:update_theta} 
According to consistency results, the larger $\theta $ is, the better
approximate solution would be. However, from a computational point of view,
with large values of $\theta $, the approximate problems are difficult and
the algorithms converge often to local minimums. 
We can overcome this bottleneck by using an update procedure for $\theta $.
Starting with a chosen value $\theta ^{0}$, at each iteration $k$, we
compute $(x^{k+1},y^{k+1})$ from $(x^{k},y^{k})$ by applying the DCA based
algorithms with $\theta =\theta ^{k}$. The sequence $\{\theta ^{k}\}_{k}$ is
increasing by $\theta ^{k+1}=\theta ^{k}+\Delta \theta ^{k}$. $\Delta \theta
^{k}$ can be fixed or updated during the iterations (see Experiment 1 in the
next section).



\section{Application to Feature selection in SVM}

\label{num}

In this section we focus on the context of Support Vector Machines learning
with two-class linear models. Generally, the problem can be formulated as
follows.

Given two finite point sets $\mathcal{A}$ (with label $+1$) and $\mathcal{B}$
(with label $-1$) in $\mathbb{R}^{n}$ represented by the matrices $A\in 
\mathbb{R}^{N_{A}\times n}$ and $B\in \mathbb{R}^{N_{B}\times n}$,
respectively, we seek to discriminate these sets by a separating hyperplane (%
$x\in \mathbb{R}^{n},b\in \mathbb{R)}$ 
\begin{equation}
P=\{w\in \mathbb{R}^{n}:w^{T}x=b\}
\end{equation}%
which uses as few features as possible. We adopt the notations introduced in 
\cite{Bradley-Mangasarian} and consider the optimization problem proposed in 
\cite{Bradley-Mangasarian} that takes the form ($e\in \mathbb{R}^{n}$ being
the vector of ones): 
\begin{equation}
\min_{x,b}(1-\lambda )\left( \frac{1}{N_{A}}\Vert \max \{0,-Ax+eb+e\}\Vert
_{1}+\frac{1}{N_{B}}\Vert \max \{0,Bx-eb+e\}\Vert _{1}\right) +\lambda
\left\Vert x\right\Vert _{0}  \label{svm1}
\end{equation}%
or equivalently 
\begin{equation}
\begin{array}{ll}
\min_{x,y,\xi ,\zeta } & (1-\lambda )(\frac{1}{N_{A}}e^{T}\xi +\frac{1}{N_{B}%
}e^{T}\zeta )+\lambda \left\Vert x\right\Vert _{0} \\ 
s.t. & -Ax+eb+e\leq \xi ,\ Bx-eb+e\leq \zeta ,\ \xi \geq 0,\ \zeta \geq 0.%
\end{array}
\label{pbmain}
\end{equation}%
The nonnegative slack variables $\xi _{j},j=1,...,N_{A}$ represent the
errors of classification of $a_{j}\in \mathcal{A}$ while $\zeta
_{j},j=1,..., N_{B}$ represent the errors of classification of $b_{j}\in 
\mathcal{B}$. More precisely, each positive value of\textbf{\ }$\xi _{j}$%
 determines the distance between a point\textbf{\ }$a_{j}\in A$%
(lying on the wrong side of the bounding hyperplane
$w^{T}x=b+1$ for $\mathcal{A)}$ and the
hyperplane itself. Similarly for $\zeta _{j}$, $\mathcal{B}$ and $w^{T}x=b-1$%
. The first term of the objective function of (\ref{pbmain}) is the average
error of classification, and the second term is the number of nonzero
components of the vector $x$, each of which corresponds to a representative
feature. Further, if an element of $x$ is zero, the corresponding feature is
removed from the dataset. Here $\lambda $ is a control parameter of the
trade-off between the training error and the number of selected features.

Observe that the problem (\ref{pbmain}) is a special case of (\ref{mainpbl0}) where the function $f$ is given by

\begin{equation}
f(x,b ,\xi ,\zeta ):=(1-\lambda )\left(\frac{1}{N_A}e^{T}\xi +\frac{1}{N_B}%
e^{T}\zeta \right)  \label{eqf}
\end{equation}%
and $K$ is a polytope defined by

\begin{equation}
K:=\left\{ (x,b,\xi ,\zeta )\in \mathbb{R}^{n}\times \mathbb{R}\times 
\mathbb{R}_{+}^{N_{A}}\times \mathbb{R}_{+}^{N_{B}}:-Ax+eb+e\leq \xi ,\
Bx-eb+e\leq \zeta \right\} .  \label{eqk}
\end{equation}%
Then the approximate problem takes the form 
\begin{equation}
\min \left\{ F(x,b,\xi ,\zeta ):=f(x,b,\xi ,\zeta )+\lambda
\sum_{i=1}^{n}r(x_{i}):(x,b,\xi ,\zeta )\in K\right\} ,  \label{prob:fsDC}
\end{equation}%
where $r$ is one of the sparsity-inducing functions given in Table \ref%
{tab:app-form}. This problem is also equivalent to 
\begin{equation}
\min \left\{ \overline{F}(x,b,\xi ,\zeta ,z):=f(x,b,\xi ,\zeta )+\lambda
\sum_{i=1}^{n}r(z_{i}):(x,b,\xi ,\zeta ,z)\in \overline{K}\right\} ,
\label{prob:fsDC-eqv}
\end{equation}%
where $\overline{K}=\left\{ (x,b,\xi ,\zeta ,z):(x,b,\xi ,\zeta )\in
K,-z_{i}\leq x_{i}\leq z_{i}~\forall i=1,\dots ,n\right\} $.

Note that, since $K$ is a polyhedral convex set, all the resulting
approximate problems \eqref{prob:fsDC} with approximation functions given in
Table~\ref{tab:DCA1} \textbf{(}except for $r=r_{PiL}$) are equivalent to the
problem \eqref{pbmain} in the sense of Corollary \ref{crl:1}. More strongly,
from Proposition \ref{prop:equiv-spe}, if $r=r_{cap}$ and $\theta
>\theta^*:= \frac{1-\lambda }{\lambda }\Delta $, where 
\begin{equation}
\Delta :=\max_{j=1,\dots ,n}\left\{ \frac{1}{N_{A}}%
\sum_{i=1}^{N_{A}}|A_{ij}|+\frac{1}{N_{B}}\sum_{i=1}^{N_{B}}|B_{ij}|\right\}
,  \label{deltasvm}
\end{equation}%
then the problems \eqref{pbmain} and \eqref{prob:fsDC} are equivalent.

Here the function $f$ is simply linear, and DC components of $f$ is taken as 
$g = f$ and $h = 0$. According to Algorithms \ref{DCA1}, \ref{DCA2}, \ref%
{DCA3} and \ref{DCA4}, DCA for solving the problem \eqref{prob:fsDC} is
described briefly as follows. \newline
\newline
\textbf{DCA1:} For $\eta$ given in Table \ref{tab:DCA1}, let $\psi (t) =
\eta |t| - r(t)$. At each iteration $k$, DCA\ref{DCA1} for solving %
\eqref{prob:fsDC} consists of

- Compute $\overline{z}_i^k \in \lambda \partial \psi(x_i^k)~\forall
i=1,\dots,n$ as given in Table \ref{tab:DCA1}.

- Compute $(x^{k+1},b^{k+1},\xi^{k+1},\zeta^{k+1})$ by solving the linear
program 
\begin{equation}  \label{Pk:fs:DCA1}
\min \left\{ (1-\lambda )\left(\frac{1}{N_A}e^{T}\xi +\frac{1}{N_B}%
e^{T}\zeta \right) + \lambda \eta \sum_{i=1}^n z_i - \langle \overline{z}%
^k,x \rangle : (x,b,\xi,\zeta,z) \in \overline{K}\right\}.
\end{equation}

Since $f$ is linear and $K$ is a polyhedral convex set, the first DC
component $G_{1}$ in \eqref{eqn:DC1} is polyhedral convex. Therefore, %
\eqref{eqn:DC1} is always a polyhedral DC program. According to the
convergence property of polyhedral DC programs, DCA\ref{DCA1} applied to %
\eqref{prob:fsDC} generates a sequence $\{(x^{k},b^{k},\xi ^{k},\zeta
^{k})\} $ that converges to a critical point $(x^{\ast },b^{\ast },\xi
^{\ast },\zeta ^{\ast })$ after finitely many iterations. Furthermore, if $%
r=r_{cap} $ and $|x_{i}^{\ast }|\neq \frac{1}{\theta }~\forall i=1,\dots ,n$%
, the second DC component $H_{1}$ in \eqref{eqn:DC1} is polyhedral convex
and differentiable at $(x^{\ast },b^{\ast },\xi ^{\ast },\zeta ^{\ast })$.
Using the DCA's convergence property v) in Sect. \ref{dca}, we deduce that $%
(x^{\ast },b^{\ast },\xi ^{\ast },\zeta ^{\ast })$ is a local solution of %
\eqref{prob:fsDC}. \newline
\newline
\textbf{DCA2:} At each iteration $k$, DCA\ref{DCA2} for solving %
\eqref{prob:fsDC} consists of

- Compute $\overline{z}^k_i \in - \lambda \partial (-r)(|x_i^k|)~\forall
i=1,\dots,n$ as given in Table \ref{tab:l1}.

- Compute $(x^{k+1},b^{k+1},\xi^{k+1},\zeta^{k+1})$ by solving the linear
program 
\begin{equation*}
\min \left\{ (1-\lambda )\left(\frac{1}{N_A}e^{T}\xi +\frac{1}{N_B}%
e^{T}\zeta \right) + \langle \overline{z}^k,z \rangle : (x,b,\xi,\zeta,z)
\in \overline{K}\right\}.
\end{equation*}

Similar to the case of DCA1 mentioned above, \eqref{eqn:DC2} is also a
polyhedral DC program. Thus, DCA\ref{DCA2} applied to \eqref{prob:fsDC-eqv}
generates a sequence $\{(x^{k},b^{k},\xi ^{k},\zeta ^{k},|x^{k}|)\}$ that
converges to a critical point $(x^{\ast },b^{\ast },\xi ^{\ast },\zeta
^{\ast },|x^{\ast }|)$ after finitely many iterations. Furthermore, if $%
r=r_{cap}$ and $|x_{i}^{\ast }|\neq \frac{1}{\theta }~\forall i=1,\dots ,n$,
the second DC component $H_{2}$ in \eqref{eqn:DC2} is polyhedral convex and
differentiable at $(x^{\ast },b^{\ast },\xi ^{\ast },\zeta ^{\ast },|x^{\ast
}|)$. Then $(x^{\ast },b^{\ast },\xi ^{\ast },\zeta ^{\ast },|x^{\ast }|)$
is a local solution of \eqref{prob:fsDC-eqv}. \newline
\newline
\textbf{DCA3:} At each iteration $k$, DCA\ref{DCA3} for solving %
\eqref{prob:fsDC} consists of

- Compute $\overline{z}^k_i \in \frac{- \lambda}{2(|x_i^k|^2+\epsilon)^{1/2}}
\partial (-r)((|x_i^k|^2+\epsilon)^{1/2})~\forall i=1,\dots,n$ as given in
Table \ref{tab:l2}.

- Compute $(x^{k+1},b^{k+1},\xi^{k+1},\zeta^{k+1})$ by solving the quadratic
convex program 
\begin{equation*}
\min \left\{(1-\lambda )\left(\frac{1}{N_A}e^{T}\xi +\frac{1}{N_B}e^{T}\zeta
\right) + \sum_{i=1}^n \overline{z}^k_i x_i^2 : (x,b,\xi,\zeta) \in
K\right\}.
\end{equation*}
\textbf{DCA4:} Consider the case $r = r_{PiL}$. At each iteration $k$, DCA%
\ref{DCA4} for solving \eqref{prob:fsDC} consists of

- Compute $\overline{z}^k_i \in \lambda \partial \psi_{PiL}(x^k_i)~\forall
i=1,\dots,n$ via \eqref{eqn:subgrad}.

- Compute $(x^{k+1},b^{k+1},\xi^{k+1},\zeta^{k+1})$ by solving the linear
program 
\begin{eqnarray*}
\min && \left\{(1-\lambda )\left(\frac{1}{N_A}e^{T}\xi +\frac{1}{N_B}%
e^{T}\zeta \right) + \frac{\lambda \theta}{a-1} \sum_{i=1}^n t_i - \langle 
\overline{z}^k, x \rangle \right\}, \\
s.t. && (x,b,\xi,\zeta,t) \in \overline{K}, \frac{1}{\theta} \le t_i
~\forall i=1,\dots,n.
\end{eqnarray*}

Since the second DC component $H_{4}$ in \eqref{eqn:DC4} is polyhedral
convex, \eqref{eqn:DC4} is a polyhedral DC program. Thus, DCA\ref{DCA4}
applied to \eqref{prob:fsDC} generates a sequence $\{(x^{k},b^{k},\xi
^{k},\zeta ^{k})\}$ that converges to a critical point $(x^{\ast },b^{\ast
},\xi ^{\ast },\zeta ^{\ast })$ after finitely many of iterations. Moreover,
if $|x_{i}^{\ast }|\neq \frac{1}{\theta }~\forall i=1,\dots ,n$, then $H_{4}$
is differentiable at $(x^{\ast },b^{\ast },\xi ^{\ast },\zeta ^{\ast })$.
This implies that $(x^{\ast },b^{\ast },\xi ^{\ast },\zeta ^{\ast })$ is a
local solution of \eqref{prob:fsDC}. \newline

The stopping criterion of our algorithms is given by 
\begin{equation*}
\|x^{k+1}-x^k\| + |b^{k+1}-b^k| + \|\xi^{k+1}-\xi^k\| +
\|\zeta^{k+1}-\zeta^k\| \le \tau (1+\|x^k\|+|b^k|+\|\xi^k\|+\|\zeta^k\|),
\end{equation*}
where $\tau$ is a small tolerance. \newline

We have seen in Sect. \ref{exact} that the approximate problem using Capped-$%
\ell _{1}$ and SCAD approximations are equivalent to the original problem if
the parameter $\theta $ is beyond a certain threshold: $\theta \geq \theta
_{0}$ (cf. Proposition \ref{cap} and Proposition \ref{propo_equi}). However,
the computation of such a value $\theta _{0}$ is in general not available,
hence one must take large enough values for $\theta _{0}$. But, as discussed
in Sect. \ref{sect:update_theta}, a large value of $\theta $ makes the
approximate problem hard to solve. For the feature selection in SVM, we can
compute exactly a $\theta _{0}$ as shown in (\ref{deltasvm}), but it is
quite large. Hence we use an updating\textbf{\ }$\theta $\textbf{\ }%
procedure. On the other hand, in the DCA1 scheme, at each iteration, we have
to compute $\bar{z}^{k}\in \partial \lambda \psi (x^{k})$ and when $\psi $
is not differentiable at $x^{k}$, the choice of $\bar{z}^{k}$ can influence
on the efficiency of the algorithm. For Capped-$\ell _{1}$ approximation,
based on the properties of this function we propose a specific way to
compute $\bar{z}^{k}$. Below, we describe the updating $\theta $ procedure
for DCA1 with Capped-$\ell _{1}$ approximation.\newline

\noindent \textbf{Initialization:} $\Delta \theta > 0, \alpha^0 = +\infty$, $%
\theta^0 = 0$, $k = 0$. Let $(x^0,b^0,\xi^0,\zeta^0)$ be a solution of the
linear problem \eqref{prob:fsDC}. \newline
\textbf{Repeat} \newline
1. $I = \{i:0 < |x^k_i| < \alpha^k\}$, $\alpha^{k+1} = 
\begin{cases}
\max\{|x^k_i|:i\in I\} & \text{if } I \ne \emptyset, \\ 
\alpha^k & \text{otherwise.}%
\end{cases}%
$ \newline
2. Compute $\theta^{k+1} = \min\left\{\theta^*, \max \left\{ \frac{1}{%
\alpha^{k+1}}, \theta^k + \Delta \theta \right\} \right\}$. \newline
3. Compute $\overline{z}^k$: For $i=1,\dots,n$

\begin{enumerate}
\item[-] If $|x^k_i| < \alpha^{k+1}$, $\overline{z}^k_i = 0$.

\item[-] If $|x^k_i| > \alpha^{k+1}$, $\overline{z}^k_i = \text{sign}(x^k_i)
\lambda \theta$.

\item[-] If $|x^k_i| = \alpha^{k+1}$, compute $F_i^-$ (resp. $F_i^+$) the
left (resp. right) derivative of the function $u(x,b)$ w.r.t. the variable $%
x_i$ at $x^k_i$, where 
\begin{equation*}
u(x,b) = (1-\lambda )\left(\frac{1}{N_A}\|\max\{0,-Ax+eb +e\}\|_1 +\frac{1}{%
N_B} \|\max\{0,Bx-eb +e\}\|_1 \right)+\lambda \sum_{j=1}^n r(x_j).
\end{equation*}
Then $\overline{z}^k_i = 
\begin{cases}
\text{sign}(x^k_i) \lambda \theta^{k+1} & \text{if } x^k_i (F_i^- + F_i^+) <
0 \\ 
0 & \text{ortherwise.}%
\end{cases}%
$
\end{enumerate}

4. Solve the linear problem \eqref{Pk:fs:DCA1} with $\eta = \theta^{k+1}$ to
obtain $(x^{k+1},b^{k+1},\xi^{k+1},\zeta^{k+1})$. \newline
5. $k=k+1$. \newline
\textbf{Until:} Convergence of $\{x^k,b^k,\xi^k,\zeta^k\}$.\newline

In the above procedure, the computation of $\overline{z}^k$ is slightly
different from formula given in Table \ref{tab:DCA1}. When $|x^k_i| =
\alpha^{k+1}$, $\partial r(x^k_i)$ is an interval. Taking into account
information of derivative of $u$ w.r.t. the variable $x_i$ at $x^k_i$ helps
us judge which between two extreme values of $\partial r(x^k_i)$ may give
better decrease of algorithm.

At each iteration, the value of $\theta$ increases at least $\Delta \theta >
0$ as long as it does not exceed $\theta^*$ -- the value from which the
problems \eqref{pbmain} and \eqref{prob:fsDC} are equivalent. Moreover, we
know that for each fixed $\theta$, DCA1 has finite convergence. Hence, the
above procedure also possesses finite convergence property.

If $F(x^{k+1},b^{k+1},\xi^{k+1},\zeta^{k+1}) = F(x^k,b^k,\xi^k,\zeta^k)$
then $(x^k,b^k,\xi^k,\zeta^k)$ is a critical point of \eqref{prob:fsDC} with 
$r = r_{cap}$ and $\theta = \theta^{k+1}$. In addition, if $\alpha^{k+1} =
\alpha^k$, which means that $|x^k_i| \ge \alpha^k \ge \frac{1}{\theta^k}$
for any $i \in supp(x^k)$, then $(x^k,b^k,\xi^k,\zeta^k)$ is a critical
point of \eqref{prob:fsDC} for all $\theta \ge \theta^{k+1}$.


\subsection{Computational experiments}

\subsubsection{Datasets}

Numerical experiments were performed on several real-word datasets taken
from well-known UCI data repository and from challenging feature-selection
problems of the NIPS 2003 datasets. In Table \ref{tab:data}, the number of
features, the number of points in training and test set of each dataset are
given. The full description of each dataset can be found on the web site of
UCI repository and NIPS 2003.

\begin{table}[htbp]
\caption{Datasets}
\label{tab:data}\centering
\begin{tabular}{llll}
\hline
Data & \#features & \# points in training set & \# points in test set \\ 
\hline
Ionosphere & 34 & 234 & 117 \\ 
WPBC (24 months) & 32 & 104 & 51 \\ 
WPBC (60 months) & 32 & 380 & 189 \\ 
Breast Cancer & 24481 & 78 & 19 \\ 
Leukemia & 7129 & 38 & 34 \\ 
Arcene & 10000 & 100 & 100 \\ 
Gisette & 5000 & 6000 & 1000 \\ 
Prostate & 12600 & 102 & 21 \\ 
Adv & 1558 & 2458 & 821 \\ \hline
\end{tabular}%
%
%
%
%
%
%
%
\end{table}


\subsubsection{Set up experiments}

All algorithms were implemented in the Visual C++ 2008, and performed on a
PC Intel i5 CPU650, 3.2 GHz of 4GB RAM. CPLEX 12.2 was used for solving
linear/quadratic programs. We stop all algorithms with the tolerance $%
\epsilon =10^{-5} $. The non-zero elements of $x$ are determined according
to whether $|x_{i}|$ exceeds a small threshold ($10^{-5}$).

For the comparison of algorithms, we are interested in the accuracy (PWCO -
Percentage of Well Classified Objects) and the sparsity of obtained solution
as well as the rapidity of the algorithms. $POWC_1$ (resp. $POWC_2$) denotes
the POWC on training set (resp. test set). The sparsity of solution is
determined by the number (and percentage) of selected features ($SF$) while
the rapidity of algorithms is measured by the CPU time in seconds.


\subsubsection{Experiment 1}

In this experiment, we study the effectiveness of the three proposed DCA
schemes DCA1, DCA2 and DCA3 for a same approximation. Capped-$\ell
_{1}$ approximation is chosen for this experiment. 
For each dataset, the same value of $\lambda $ is used for all algorithms.
We set $\lambda =0.1$ for first three datasets (\textit{Ionosphere,
WPBC(24), WPBC(60)}) while $\lambda =0.001$ is used for five large datasets (%
\textit{Adv, Arcene, Breast, Gisette, Leukemia}). 
To chose a suitable value of $\theta $ for each algorithm DCA1, DCA2 and DCA3, we perform
them by $10$ folds cross-validation procedure on the set $\{0.001,0.005,0.01,0.1,0.5,1,2,3,5,10,20,50,100,500\}$
and then take the value corresponding to the best results.
Once $\theta$ is chosen (its value is given in   Table \ref{tab:CompareDCA1-2-3}),  
 we perform  these algorithms  $10$ times from $10$ random starting solutions and report,
in the columns 3 - 5 of   Table \ref{tab:CompareDCA1-2-3},
the mean and standard deviation of the accuracy, the sparsity of obtained solutions and CPU time of the algorithm.

We are also interested on
the efficiency of Updating $\theta $ procedure. For this purpose, we compare
two versions of DCA1    - with and without Updating $\theta $ procedure
(in case of Capped-$\ell
_{1}$ approximation).
For a fair comparison, we first
run DCA1 with Updating $\theta $ procedure and then perform DCA1 with the fixed value $\theta ^{\ast }$
which is
the last value of $\theta $ when the Updating $\theta $ procedure stops.
Computational results are reported in the columns 6 (DCA1 with fixed $\theta $) and $7$ (DCA1 with Updating $\theta $ procedure) of   Table \ref{tab:CompareDCA1-2-3}.

To evaluate the globality of the DCA based algorithms we use CPLEX 12.2  for
globally solving the exact formulation problem (\ref{l0 obj penal}) via exact penalty techniques
(Mixed 0-1 linear programming problem) and report the results in the last column of Table \ref{tab:CompareDCA1-2-3}.

Bold
values in the result tables correspond to best results for each data instance.

\begin{table}[tbh]
\caption{Comparison of different DCA schemes for Capped-$\ell_1$
approximation}
\label{tab:CompareDCA1-2-3}\centering
\setlength{\tabcolsep}{0.2cm} {\scriptsize {\ 
\begin{tabular}{llllllll}
\hline
&  & DCA1 & DCA2 & DCA3 & DCA1 & DCA1 with  & CPLEX \\ 
&  &  &  &  & with $\theta^*$ & Updating $\theta$  &  \\ \hline
Ionosphere & $\theta$ & 3 & 5 & 3 & 4,3 & 4,3 &  \\ 
& $POWC_1$ & 86,2 $\pm$1,5 & 85,2 $\pm$1,7 & 84,8 $\pm$1,8 & 84,0 $\pm$1,2 & 
\textbf{90,2} & \textbf{90,2} \\ 
& $POWC_2$ & 80,3 $\pm$1,6 & 75,3 $\pm$1,3 & 74,3 $\pm$1,3 & 80,3 $\pm$1,4 & 
\textbf{83,7} & \textbf{83,7} \\ 
& FS & 3,5 (10,3\%) & 3,8 (11,2\%) & 3,8 (11,2\%) & 3,2 (9,4\%) & \textbf{2
(5,9\%)} & \textbf{2 (5,9\%)} \\ 
& CPU & \textbf{0,2} & \textbf{0,2} & 0,7 & 0,3 & 0,6 & 2,5 \\ 
WPBC(24) & $\theta$ & 1 & 0,1 & 0,1 & 661 & 661 &  \\ 
& $POWC_1$ & \textbf{84,3 $\pm$1,4} & 75,3 $\pm$1,3 & 77,4 $\pm$1,1 & 75,3 $%
\pm$1,2 & 77,4 & 77,4 \\ 
& $POWC_2$ & 77,9 $\pm$1,4 & \textbf{80,2 $\pm$1,6} & 79,3 $\pm$1,6 & 72,3 $%
\pm$1,2 & 77,2 & 78,4 \\ 
& FS & 7,4 (23,1\%) & 8,5 (26,6\%) & 8,5 (26,6\%) & 8,4 (26,3\%) & 8 (25,0\%)
& \textbf{7 (21,9\%)} \\ 
& CPU & \textbf{0,2} & 0,3 & 0,8 & \textbf{0,2} & 1,1 & 6,4 \\ 
WPBC(60) & $\theta$ & 1 & 3 & 3 & 347 & 347 &  \\ 
& $POWC_1$ & 96,2 $\pm$1,3 & 95,2 $\pm$1,3 & 95,2 $\pm$1,3 & \textbf{98,2 $%
\pm$1,3} & 96 & \textbf{96} \\ 
& $POWC_2$ & 92,5$\pm$1,4 & 92,5$\pm$1,4 & 90,8$\pm$1,8 & \textbf{96,8$\pm$%
1,8} & 95,3 & \textbf{95,3} \\ 
& FS & 4,7 (15,7\%) & 5,5 ( 18,3\%) & 5,7 (19,0\%) & 8,9 (29,7\%) & \textbf{%
3 (10,0\%)} & \textbf{3 (10,0\%)} \\ 
& CPU & \textbf{0,4} & 0,6 & 1,6 & 0,5 & 1 & 1,8 \\ 
Breast & $\theta$ & 5 & 10 & 2 & 435 & 435 &  \\ 
& $POWC_1$ & 95,1$\pm$1,3 & 94,2$\pm$1,3 & 95,2$\pm$1,4 & 93,2$\pm$1,6 & 
\textbf{96,8} & N/A \\ 
& $POWC_2$ & 68,3$\pm$1,2 & 67,3$\pm$1,2 & \textbf{70,3$\pm$1,6} & 66,3$\pm$%
1,1 & 65,1 & N/A \\ 
& FS & 32,6 (0,1\%) & 47,5 (0,2\%) & 43,5 (0,2\%) & 52,3 (0,2\%) & \textbf{%
28 (0,1\%)} & N/A \\ 
& CPU & 30 & \textbf{25} & 78 & 79 & 76 & 3600 \\ 
Leukemia & $\theta$ & 5 & 5 & 5 & 178 & 178 &  \\ 
& $POWC_1$ & \textbf{100} & \textbf{100} & \textbf{100} & \textbf{100} & 
\textbf{100} & N/A \\ 
& $POWC_2$ & \textbf{97,2$\pm$0,4} & 97,1$\pm$0,4 & 96,8$\pm$0,3 & 94,8$\pm$%
0,7 & \textbf{97,2} & N/A \\ 
& FS & 8,2 (0,1\%) & 8,5 (0,1\%) & 8,5 (0,1\%) & 12,0 (0,2\%) & \textbf{8
(0,1\%)} & N/A \\ 
& CPU & \textbf{10} & \textbf{10} & 75 & 14 & 17 & 3600 \\ 
Arcene & $\theta$ & 0,1 & 0,01 & 3 & 328 & 328 &  \\ 
& $POWC_1$ & \textbf{100} & \textbf{100} & \textbf{100} & \textbf{100} & 
\textbf{100} & N/A \\ 
& $POWC_2$ & 80$\pm$1,6 & \textbf{82$\pm$1,1} & 81$\pm$1,9 & 61$\pm$1,1 & 70
& N/A \\ 
& FS & 78,5 ( 0,79\%) & 82,4 (0,82\%) & 82,4 (0,82\%) & 35 (0,35\%) & 
\textbf{32 (0,32\%)} & N/A \\ 
& CPU & \textbf{21} & 26 & 273 & 30 & 118 & 3600 \\ 
Gisette & $\theta$ & 0,1 & 0,01 & 0,1 & 735 & 735 &  \\ 
& $POWC_1$ & \textbf{92,5$\pm$1,3} & 88,5$\pm$1,3 & 88,5$\pm$1,3 & 90,5$\pm$%
1,2 & 91,2 & N/A \\ 
& $POWC_2$ & \textbf{85,3$\pm$1,2} & 83,4$\pm$1,2 & 83,1$\pm$1,6 & 84,1$\pm$%
1,1 & 83,2 & N/A \\ 
& FS & 339,4 (6,8\%) & 330,7 (6,6\%) & 332,2 (6,6\%) & 456 (9,1\%) & \textbf{%
123 (2,5\%)} & N/A \\ 
& CPU & 87 & \textbf{65} & 253 & 71 & 387 & 3600 \\ 
Adv & $\theta$ & 0,1 & 0,01 & 0,1 & 321 & 321 &  \\ 
& $POWC_1$ & 95,5$\pm$1,5 & 92,3$\pm$1,5 & 95,3$\pm$1,5 & 92,3$\pm$1,2 & 
\textbf{97,2} & N/A \\ 
& $POWC_2$ & \textbf{94,2$\pm$1,1} & 93,2$\pm$1,5 & 93,1$\pm$1,2 & 92,1$\pm$%
1,6 & 93,2 & N/A \\ 
& FS & 5,4 (0,35\%) & 6,2 (0,40\%) & 6,4 (0,41\%) & 6,5 (0,42\%) & \textbf{5
(0,32\%)} & N/A \\ 
& CPU & \textbf{2,1} & 2,4 & 7,8 & 2,3 & 4,6 & 3600 \\ \hline
\end{tabular}%
} }
\end{table}

\medskip

\noindent {\bf   Comments on numerical results}

\begin{itemize}
\item Comparison between DCA1, DCA2 and DCA3 (columns $3$ - $5$)

\begin{itemize}
\item Concerning the correctness, DCA1 furnishes the best solution out of
the three algorithms for all datasets (with an important gain of $6,9\%$ on
dataset WPBC(24)). DCA2 and DCA3 are comparable in terms of correctness.

\item As for the sparsity of solution, all the three DCA schemes reduce
considerably the number of selected features (up to $99\%$ on large datasets
such as \textit{Arcene}, \textit{Breast}, \textit{Leukemia}, \ldots).
Moreover, DCA1 gives better results than DCA2/DCA3 on $6$ out of $7 $
datasets.

\item In terms of CPU Time, DCA1 and DCA2 are faster than DCA3. This is
natural, since at each iteration, the first two algorithms only require solving one
linear program while DCA3 has to solve one convex quadratic program. DCA1 is
somehow a bit faster than DCA2 on $5$ out $7$ datasets.

\item Overall, we see that DCA1 is better than DCA2 and DCA3 on all the three evaluation criteria.
Hence, it seems to be that the first DCA scheme  is more appropriate
than the other two for Capped-$\ell_1$ approximation.

\end{itemize}

\item DCA1 with and without Updating $\theta$ procedure (columns $3$, $6$ and $7$):

\begin{itemize}
\item For all datasets, Updating $\theta $ procedure gives a better solution
(on both accuracy and sparsity) than DCA1 with $\theta =\theta
^{*}$. 

\item Except for dataset \textit{WPBC(24)}, Updating $\theta $ procedure is
better than DCA1 with $\theta $ chosen by $10$ folds cross-validation in
terms of sparsity of solution. As for accuracy, the two algorithms are
comparable.

\item  The choice of the value of $\theta$ defining the approximation
function is very important. Indeed, the  results given in columns $3$ and $6$ are  far different,
due to the fact that, the value of $\theta $ chosen by $10$
folds cross-validation is much more smaller than $\theta ^{\ast }$.
These results confirm our analysis in Subsection~\ref{sect:update_theta} above:
while the approximate function would be better with larger values of $\theta$,
the approximate problems become more difficult and it can be happened that the obtained solutions
are worse when $\theta$ is quite large. To overcome this ''contradiction'' between theoretical
and computational aspects, the proposed Updating $\theta$ procedure seems to be efficient.
  
\end{itemize}

\item Comparison between DCA based algorithms and CPLEX for solving
the original problem (\ref{l0 obj penal})

\begin{itemize}
\item For \textit{Ionosphere} and \textit{WPBC(60)}, Updating $\theta $
procedure for Capped-$\ell _{1}$ gives exactly the same accuracy and the
same number of selected features as CPLEX. It means that Updating $\theta $
procedure reaches the global solution for those two datasets. For \textit{%
WPBC(24)}, the two obtained solutions are slightly different (same accuracy
on training set and $7$ selected features for CPLEX instead of $8$ for
Updating $\theta $ procedure).

\item For large datasets, CPLEX can't furnish a solution with a CPU Time
limited to $3600$ seconds while DCA based algorithms give a good
solution in a short time.
\end{itemize}
\end{itemize}


\subsubsection{Experiment 2}

In the second experiment, we study the effectiveness of different
approximations of $\ell_0$. We use DCA1 for all approximations except PiL
for which DCA4 is applied (cf. Section \ref{new_approximation}).

In this experiment, for the trade-off parameter $\lambda $, we used the
following set of candidate values \linebreak $%
\{0.001,0.002,0.003,0.004,0.05,0.1,0.25,0.4,0.7,0.9\}$. The value of
parameter $\theta $ is chosen in the set \linebreak $%
\{0.001,0.005,0.01,0.1,0.5,1,2,3,5,10,20,50,100,500\}$. The second parameter 
$a$ of SCAD approximation is taken from $\{1,2,3,5,10,20,30,50,100\}$. For
each algorithm, we firstly perform a $10$-folds cross-validation to
determine the best set of parameter values. In the second step, we run each
algorithm, with the chosen set of parameter values in step 1, $10$ times
from $10$ starting random points and report the mean and standard deviation
of each evaluation criterion. The comparative results are reported in Table \ref%
{tab:CompareApproximation}.

\begin{table}[tbh]
\caption{Comparison of different approximations}
\label{tab:CompareApproximation}\centering
\setlength{\tabcolsep}{0.2cm} {\scriptsize {\ \centering
\begin{tabular}{lllllllll}
\hline
&  & DCA1 & DCA1 & DCA1 & DCA1 & DCA1 & DCA1 & DCA4 \\ 
&  & Capped-l1 & SCAD & Exp & lp+ & lp- & Log & PiL \\ \hline
Ionosphere & $POWC_1$ & \textbf{86,2 $\pm$1,5} & 80,1 $\pm$1,4 & 82,1 $\pm$%
1,4 & 81,5 $\pm$1,3 & 83,1 $\pm$1,4 & 81,2 $\pm$1,4 & 83,2 $\pm$1,4 \\ 
& $POWC_2$ & 80,3 $\pm$1,6 & 73,5 $\pm$1,6 & \textbf{84,8 $\pm$1,3} & 75,1 $%
\pm$1,1 & 70,3 $\pm$1,2 & 73,1 $\pm$2,1 & 83,5 $\pm$1,6 \\ 
& SF & 3,5 (10,3\%) & 3,1 (9,1\%) & \textbf{2,3 (6,8\%)} & 3,8 (11,2\%) & 
3,1 (9,1\%) & 3,3 (9,7\%) & 2,6 (7,6\%) \\ 
& CPU & \textbf{0,2} & 0,3 & 0,3 & \textbf{0,2} & 0,3 & 0,15 & \textbf{0,2}
\\ 
WPBC(24) & $POWC_1$ & \textbf{84,3 $\pm$1,4} & 77 $\pm$1,3 & \textbf{84,3 $%
\pm$1,5} & 81,3 $\pm$1,2 & 81,9 $\pm$1,2 & 71,3 $\pm$1,4 & 84,2 $\pm$1,4 \\ 
& $POWC_2$ & 77,9 $\pm$1,4 & \textbf{79,3 $\pm$1,6} & 74,3 $\pm$1,9 & 78,4 $%
\pm$1,2 & \textbf{79,8 $\pm$1,1} & 68,4 $\pm$1,6 & 78,5 $\pm$1,4 \\ 
& SF & 7,4 ( 23,1\%) & 8,1 (25,3\%) & \textbf{7,2 (22,5\%)} & 7,8 (24,4\%) & 
7,5 (23,4\%) & 7,2 (22,5\%) & 7,6 (23,8\%) \\ 
& CPU & \textbf{0,1} & 0,2 & \textbf{0,1} & 0,2 & 0,2 & 0,2 & 0,2 \\ 
WPBC(60) & $POWC_1$ & \textbf{97,2 $\pm$1,3} & 93,5 $\pm$1,7 & 95,1 $\pm$1,6
& 93 $\pm$1,2 & 94,5 $\pm$1,1 & 89 $\pm$1,5 & 95,2 $\pm$1,3 \\ 
& $POWC_2$ & \textbf{93,5$\pm$1,4} & 89,1 $\pm$1,9 & 92,3 $\pm$1,9 & 85 $\pm$%
1,2 & 90,6 $\pm$1,2 & 80 $\pm$1,6 & 88,5$\pm$1,1 \\ 
& SF & 5,4 (18,0\%) & \textbf{5,2 (17,3\%)} & \textbf{5,2 (17,3\%)} & 5,9
(19,7\%) & 5,7 (19,0\%) & 5,4 (18,0\%) & 5,4 (18,0\%) \\ 
& CPU & \textbf{0,4} & \textbf{0,4} & \textbf{0,4} & 0,5 & \textbf{0,4} & 0,6
& 0,5 \\ 
Breast & $POWC_1$ & \textbf{98,7$\pm$1,3} & 91,9$\pm$1,4 & 96,3$\pm$1,4 & 
93,2$\pm$1,4 & 91,9$\pm$1,4 & 91,2$\pm$1,4 & 92,4$\pm$1,2 \\ 
& $POWC_2$ & 68,3$\pm$1,2 & 69,1$\pm$1,6 & \textbf{70\%$\pm$1,4} & 67,3$\pm$%
1,1 & 69,1$\pm$1,6 & 66,3$\pm$1,2 & 71,3$\pm$1,4 \\ 
& SF & 35,3 (0,1\%) & 37,0 (0,2\%) & 37,4 ((0,2\%) & 40,3 ((0,2\%)) & 37,0
(0,2\%) & 45,3 (0,2\%) & \textbf{26,5 (0,1\%)} \\ 
& CPU & 30 & 31 & \textbf{25} & 32 & 31 & 31 & 31 \\ 
Leukemia & $POWC_1$ & \textbf{100} & 98,3 & \textbf{100} & \textbf{100} & 
98,3 & \textbf{100} & \textbf{100} \\ 
& $POWC_2$ & \textbf{97,2$\pm$0,4} & 88,3$\pm$0,6 & 97,2$\pm$0,5 & 90,1$\pm$%
0,8 & 92,3$\pm$0,6 & 90,1$\pm$0,3 & 89,2$\pm$0,9 \\ 
& SF & \textbf{8,2 (0,1\%)} & \textbf{8,2 (0,1\%)} & 8,3 (0,1\%) & 27,9
(0,4\%) & \textbf{8,2 (0,1\%)} & 27,3 (0,4\%) & 12,8 (0,2\%) \\ 
& CPU & 25 & 21 & 23 & 27 & 21 & 28 & 22 \\ 
Arcene & $POWC_1$ & \textbf{100} & \textbf{100} & \textbf{100} & \textbf{100}
& \textbf{100} & \textbf{100} & \textbf{100} \\ 
& $POWC_2$ & \textbf{80$\pm$1,6} & 78,2$\pm$1,9 & 78,9$\pm$1,4 & 78,9$\pm$1,1
& 74,2$\pm$1,2 & 72,9$\pm$1,6 & 79$\pm$1,2 \\ 
& SF & 78,5(0,79\%) & 72,5 (0,73\%) & \textbf{69,4 (0,69\%)} & 71,1 (0,71\%)
& 73,1 (0,73\%) & 72,3 (0,72\%) & 83,5 (0,84\%) \\ 
& CPU & \textbf{21} & 31 & 34 & 31 & 31 & 30 & 23 \\ 
Gisette & $POWC_1$ & \textbf{92,5$\pm$1,3} & 87,3$\pm$1,5 & 87,3$\pm$2,1 & 
88,3$\pm$2,4 & 86,4$\pm$1,2 & 86,3$\pm$2,1 & 89,5$\pm$1,4 \\ 
& $POWC_2$ & \textbf{85,3$\pm$1,2} & 81,2$\pm$1,4 & 82,2$\pm$1,2 & 77,3$\pm$%
1,3 & 82,2$\pm$1,5 & 79,3$\pm$1,4 & 84,5$\pm$1,2 \\ 
& SF & 339,4 (6,8\%) & 340,1 (6,8\%) & \textbf{330,1 (6,6\%)} & 341,5 (6,8\%)
& 342,3 (6,8\%) & 354,5 (7,8\%) & 344,3 (6,9\%) \\ 
& CPU & 87 & 81 & 98 & 102 & 81 & 102 & \textbf{72} \\ 
Adv & $POWC_1$ & \textbf{95,5$\pm$1,5} & 94,2$\pm$1,3 & \textbf{95,5$\pm$1,1}
& 93,2$\pm$1,1 & 92,2$\pm$1,5 & 95,2$\pm$1,6 & 94,1$\pm$1,8 \\ 
& $POWC_2$ & 94,2$\pm$1,1 & 94,4$\pm$1,9 & \textbf{94,5$\pm$1,5} & 80,2$\pm$%
1,5 & 88,1$\pm$1,2 & 92,2$\pm$1,5 & 90,2$\pm$1,1 \\ 
& SF & 5,4 (0,35\%) & 8,1 (0,52\%) & \textbf{5,1 (0,33\%)} & 12,3 (0,79\%) & 
6,4 (0,41\%) & 21,3 (2,8\%) & 7,4 (0,47\%) \\ 
& CPU & \textbf{2,1} & 2,5 & 2,3 & 2,8 & 2,5 & 2,8 & 3,1 \\ \hline
\end{tabular}%
} }
\end{table}
%
We observe that:

\begin{itemize}
\item In terms of sparsity of solution, the quality of all approximations
are comparable. All the algorithms reduce considerably the number of
selected features, especially for $5$ large datasets (\textit{Adv, Arcene,
Breast, Gisette, Leukemia}). For \textit{Breast} dataset, our algorithms
select only about thirty features out of $24481$ while preserving very good
accuracy (up to $98,7\%$ correctness on train set).

\item Capped-$\ell _{1}$ is the best in terms of accuracy: it gives best
accuracy on all train sets and $4$ out of $7$ test sets. The quality of
other approximations are comparable.

\item The CPU time of all the algorithms is quite small: less than $34$ seconds
(except for  \textit{Gisette}, CPU time of DCAs varies from $72$ to $102$ seconds).
\end{itemize}


\section{ Conclusion}

\label{conclu}

We have intensively studied DC programming and DCA for sparse optimization
problem including the zero-norm in the objective function. DC approximation
approaches have been investigated from both a theoretical and an algorithmic
point of view. Considering a class of DC approximation functions of the
zero-norm including all usual sparse inducing approximation functions, we
have proved several novel and interesting results: the consistency between
global (resp. local) minimizers of the approximate problem and the original
problem, the equivalence between these two problems (in the sense that, for
a sufficiently large related parameter, any optimal solution to the
approximate problem solves the original problem) when the feasible set is a
bounded polyhedral convex set and the approximation function is concave, the
equivalence between Capped-$\ell _{1}$ (and/or SCAD) approximate problems
and the original problem with sufficiently large parameter $\theta $ (in the
sense that they have the same set of optimal solutions), the way to compute
such parameters $\theta $ in some special cases, and a comparative analysis
between usual sparse inducing approximation functions. Considering the three
DC formulations for a common model to all concave approximation functions we
have developed three DCA schemes and showed the link between our algorithms
with standard approaches. It turns out that all standard nonconvex
approximation algorithms are special versions of our DCA based algorithms. A
new DCA scheme has been also investigated for the DC approximation
(piecewise linear) which is not concave as usual sparse inducing functions.
Concerning the application to feature selection in SVM, among the four DCA
schemes, three (resp. one) require solving one linear (resp. convex
quadratic) program at each iteration and enjoy interesting convergence
properties (except Algorithm 3): they converge after finitely many
iterations to a local solution in almost all cases. Numerical experiments
confirm the theoretical results: the Capped-$\ell _{1}$ has been identified
as the \textquotedblright winner\textquotedblright\ among sparse inducing
approximation functions.

Our unified DC programming framework shed a new light on sparse nonconvex
programming. It permits to establish the crucial relations among existing
sparsity-inducing methods and therefore to exploit, in an elegant way, the
nice effect of DC decompositions of objective functions. The four algorithms
can be viewed as an $\ell _{1}$-perturbed algorithm / reweighted-$\ell _{1}$
algorithm (intimately related to the $\ell _{1}$-penalized LASSO approach / reweighted-$\ell _{2}$
algorithm in case of convex objective functions. It
specifies the flexibility/versatility of these theoretical and algorithmic
tools. These results should enhance deeper developments of DC programming
and DCA, in order to efficiently model and solve real-world nonconvex sparse
optimization problems, especially in the large-scale setting.




\begin{thebibliography}{99}
\bibitem[E. Amaldi and V. Kann (1998)]{Amaldi} Amaldi, E. \& Kann, V.
(1998). On the approximability of minimizing non zero variables or
unsatisfied relations in linear systems. \textit{Theoretical Computer Science%
}, 209, 237--260.

\bibitem[Bajawa et al. (2006)]{Bajwa06} Bajwa, W., Haupt, J., Sayeed A. \&
Nowak, R. (2006). Compressive wireless sensing. \textit{Proceedings of Fifth
Int. Conf. on Information Processing in Sensor Networks}, 134--142,.

\bibitem[Baron et al. (2006)]{Baron06} Baron, D., Wakin, M.B., Duarte, M.F.,
Sarvotham, S. \& Baraniuk, R.G (2009). Distributed compressed sensing.
Technical Report ECE06-12, Electrical and Computer Engineering Department,
Rice University, November 2006.

\bibitem[Bradley and Mangasarian (1998)]{Bradley-Mangasarian} Bradley, P.S
\& Mangasarian, O.L. (1998). Feature Selection via concave minimization and
support vector machines. \textit{Proceeding of International Conference on
Machina Learning ICML'98}.

\bibitem[Bradley et al. (1998)]{Bradley98b} Bradley, P.S., Mangasarian, O.L.
\& Rosen, J.B (1998). Parsimonious Least Norm Approximation. \textit{Comput.
Optim. Appl.}, 11(1), 5--21.


\bibitem[Candes and Tao (2005)]{Candes05} Candes, E.J \& Tao, T. (2005).
Decoding by linear programming. \textit{IEEE Trans. Inf. Theory}, 51(12),
4203--4215.

\bibitem[Candes and Randhall (2006)]{Candes06} Candes, E.J. \& Randall, P.
(2006). Highly robust error correction by convex programming. \textit{IEEE
Trans. Inform. Theory}, 54, 2829--2840.

\bibitem[Candes et al. (2008)]{Candes08} Candes, E.J., Wakin, M. \& Boyd, S.
(2008). Enhancing sparsity by reweighted-$l_{1}$ minimization. \textit{J.
Four. Anal. and Appli.}, 14, 877--905.

\bibitem[Chartrand and Yin (2008)]{Chartrand08} Chartrand, R. \& Yin, W.
(2008). Iteratively reweighted algorithms for compressive sensing, ICASSP
2008.

\bibitem[Chan et al. (2007)]{Chan07} Chan, A.B., Vasconcelos, N. \&
Lanckriet, R.G. (2007). Direct Convex Relaxations of Sparse SVM. \textit{%
Proceeding ICML'07 Proceedings of the 24th international conference on
Machine learning}, 145--153.



\bibitem[Chen et al. (2010)]{XIA10} Chen, X., Xu, F.M. \& Ye, Y. (2010).
Lower bound theory of nonzero entries in solutions of l2-lp minimization. 
\textit{SIAM J. Sci. Comp.}, 32:(5), 2832--2852.


\bibitem[Collober et al. (2006)]{COLO06} Collobert, R., Sinz, F., Weston, J.
\& Bottou, L. (2006). Trading Convexity for Scalability. \textit{Proceedings
of the 23th International Conference on Machine Learning (ICML 2006)},
Pittsburgh, PA.



\bibitem[Fan and Li (2001)]{FAN01} Fan, J. \& Li, R. (2001). Variable
selection via nonconcave penalized likelihood and its oracle properties. 
\textit{J. Amer. Stat. Ass.}, 96(456), 1348--1360.

\bibitem[Fawzi et al. (2014)]{Fawzi14} Fawzi, A., Davies, M., \& Frossard,
P. (2014). Dictionary learning for fast classification based on
soft-thresholding. submitted to \textit{International Journal of Computer
Vision}, http://arxiv.org/abs/1402.1973.

\bibitem[Fu (1998)]{Fu98} Fu, W.J. (1998). Penalized regression: the bridge
versus the lasso. \textit{J. Comp. Graph. Stat.}, 7, 397--416.

\bibitem[Gasso et al. (2009)]{GAS09} Gasso, G., Rakotomamonjy, A. \& Canu,
S. (2009). Recovering sparse signals with a certain family of nonconvex
penalties and dc programming. \textit{IEEE Trans. Sign. Proc.}, 57,
4686--4698.


\bibitem[Gribonval and Nielsen (2003)]{Gribonval03} Gribonval, R. \&
Nielsen, M. (2003). Sparse representation in union of bases. \textit{IEEE
Trans. on Information Theory}, 49, 3320--3325.

\bibitem[Gorodnitsky and Rao (1997)]{Rao97} Gorodnitsky, I.F. \& Rao, B.D.
(1997). Sparse signal reconstructions from limited data using FOCUSS: A
re-weighted minimum norm algorithm. \textit{IEEE Trans. Signal Processing},
45, 600--616.

\bibitem[Guyon el al. (2002)]{Guyon02} Guyon, I., Weston, J. Barnhill, S. \&
Vapnik, V.N. (2002). Gene selection for cancer classification using support
vector machines. \textit{Machine Learning}, 46(1--3), 389--422.

\bibitem[Guyon el al (2006)]{Guyon06} Guyon, I., Gunn, S., Nikravesh, M.\&
Zadeh, L.A. (2006). Feature extraction, foundations and applications.
Berlin:Springer.

\bibitem[Guan and Gray (2013)]{GUAN-GRAY13} Guan, W. \& Gray, A. (2013).
Sparse high-dimensional fractional-norm support vector machine via DC
programming. \textit{Computational Statistics and Data Analysis} 67,
136--148.


\bibitem[Hastie et al. (2009)]{HAS09} Hastie, T., Tibshirani, R. \&
Friedman, J. (2009). The elements of statistical learning.
Springer,Heidelberg, 2 edition.

\bibitem[Huang et al. (2008)]{HUA08} Huang, J., Horowitz, J. \& Ma, J.
(2008). Asymptotic properties of bridge estimators in sparse
high-dimensional regression models. \textit{Ann. Stat.}, 36, 587--613.

\bibitem[Knight and Fu (2000)]{KNI00} Knight, K. \& Fu, W. (2000).
Asymptotics for lasso-type estimators. \textit{Ann. Stat.}, 28, 1356--1378.

\bibitem[Krause and Singer (2004)]{Krause:icml04} Krause, N. \& Singer, Y.
(2004). Leveraging the margin more carefully. \textit{Proceedings of the
21st International Conference on Machine Learning ICML 2004}. Banff,
Alberta, Canada.

\bibitem[Peleg and Meir (2008)]{Peleg08} Peleg, D. \& Meir, R. (2008). A
bilinear formulation for vector sparsity optimization. \textit{Signal
Processing}, 8(2), 375--389.

\bibitem[Le et al. (2013)]{Block2013} Le, H.M., Le Thi, H.A., Pham Dinh, T.
\& Huynh, V.N. (2013). Block clustering based on difference of convex
functions (DC) programming and DC algorithms. \textit{Neural Computation},
25(10), 2776--807.

\bibitem[Le et al. (2013)]{Leetal2013} Le, H.M., Le Thi H.A. \& Nguyen, M.C.
(2013). DCA based algorithms for feature selection in Semi-Supervised
Support Vector Machines. \textit{Machine Learning and Data Mining in Pattern
Recognition}, Petra Perner (Ed), LNAI 7988, 528--542

\bibitem{lethi-website} Le Thi, H.A.. DC Programming and DCA,
http://lita.sciences.univ-metz.fr/$\sim$lethi.

\bibitem[Le Thi (1997)]{lethithesis} Le Thi, H.A. (1997). Contribution \`{a}
l'optimisation non convexe et l'optimisation globale: Th\'{e}orie,
Algorithmes et Applications. Habilitation \`{a} Diriger des Recherches,
Universit\'{e} de Rouen.

\bibitem[Le Thi and Pham Dinh (1997)]{PLT97} Le Thi, H.A. \& Pham Dinh, T.
(1997). Solving a class of linearly constrained indefinite quadratic
problems by DC algorithms. \textit{Journal of Global Optimization}, 11(3),
253--285.


\bibitem[Le Thi (2000)]{LT2000} Le Thi, H.A. (2000). An efficient algorithm
for globally minimizing a quadratic function under convex quadratic
constraints. \textit{Mathematical Programming}, 87:3, 401-426.

\bibitem[Le Thi and Pham Dinh (2002)]{COCONUT} Le Thi, H.A. \& Pham Dinh, T.
(2002). DC Programming: Theory, Algorithms and Applications. The State of
the Art (28 pages). Proceedings of The First International Workshop on
Global Constrained Optimization and Constraint Satisfaction (Cocos' 02),
Valbonne-Sophia Antipolis, France, October 2-4.

\bibitem[Le Thi et al. (2002)]{EJOR2002} Le Thi, H.A., Pham Dinh, T. \& and
Nguyen Van, T. (2002). Combination between Local and Global methods for
solving an Optimization problem over the Efficient set. \textit{European
Journal of Operational Research}, 142, 258-270.

\bibitem[Le Thi et Pham Dinh (2003)]{SIOPT2} Le Thi, H.A., Pham Dinh,
T.(2003). Large Scale Molecular Optimization From Distance Matrices by a
D.C. Optimization Approach. \textit{SIAM Journal on Optimization}, 4:1,
77-116.

\bibitem[Le Thi and Pham Dinh (2005)]{LTP05} Le Thi, H.A. \& Pham Dinh, T.
(2005). The DC (difference of convex functions) Programming and DCA
revisited with DC models of real world nonconvex optimization problems. 
\textit{Annals of Operations Research}, 133, 23--46.

\bibitem[Le Thi et al. a (2006)]{LeThietal2006a} Le Thi, H.A., Belghiti, T.
\& Pham Dinh, T. (2006). A new efficient algorithm based on DC programming
and DCA for Clustering. \textit{Journal of Global Optimization}, 37,
593--608.

\bibitem[Le Thi et al. b (2006)]{LeThietal2006b} Le Thi, H.A., Le H.M.\&
Pham Dinh, T. (2006). Optimization based DC programming and DCA for
Hierarchical Clustering. \textit{European Journal of Operational Research},
183, 1067--1085.

\bibitem[Le Thi et al. a (2007)]{Concave costEJOR07} Le Thi, H.A., Nguyen,
T.P. \& Pham Dinh, T. (2007). A continuous approach for solving the concave
cost supply problem by combining DCA and B\&B techniques. \textit{European
Journal of Operational Research}, 183, 1001--1012.

\bibitem[Le Thi et al. b (2007)]{Multicast} Le Thi, H.A., LE, H.M. \& Pham
Dinh, T. (2007). Optimization based DC programming and DCA for Hierarchical
Clustering. \textit{European Journal of Operational Research}, 183,
1067--1085.

\bibitem[Le Thi and Pham Dinh (2008)]{DiscreteMath08} Le Thi, H.A. \& Pham
Dinh, T. (2008). A continuous approach for the concave cost supply problem
via DC Programming and DCA, \textit{Discrete Applied Mathematics}, 156,
325--338.

\bibitem[Le Thi et al. a (2008)]{LeThietal2008a} Le Thi, H.A., Le H.M.,
Nguyen, V.V \& Pham Dinh, T. (2008). A dc programming approach for feature
selection in support vector machines learning. \textit{Journal of Advances
in Data Analysis and Classification}, 2, 259--278.

\bibitem[Le Thi et al. a (2009)]{LeThietal2009} Le Thi, H.A., Nguyen, V.V.
\& Ouchani, S. (2009). Gene Selection for Cancer Classification Using DCA. 
\textit{Journal of Fonctiers of Computer Science and Technology}, 3:6, 62--2.

\bibitem[Le Thi et al. b (2009)]{ConvergenceDCA} Le Thi, H.A., Huynh, V.N.,
\& Pham Dinh, T. (2009). Convergence Analysis of DC Algorithms for DC
programming with subanalytic data. Research Report, National Institute for
Applied Sciences, Rouen 2009
http://www.optimization-online.org/DB\_HTML/2013/08/3996.html.

\bibitem[Le Thi (2012)]{LT12} Le Thi, H.A. (2012). A new approximation for
the $\ell _{0}$-norm. Research Report LITA EA 3097, University of Lorraine.

\bibitem[Le Thi et al. a (2012)]{LTPN12} Le Thi H.A., Huynh, V.N. \& Pham
Dinh, T. (2012). Exact Penalty and Error Bounds in DC Programming. \textit{%
Journal ofGlobal Optimization}, 52(3), 509--535.

\bibitem[Le Thi et al. b (2012)]{Lethietal2012b} Le Thi H.A., Tran, D.Q. \&
Pham Dinh, T. (2012). A DC programming approach for a class of bilevel
programming problems and its application in Portfolio Selection. \textit{%
Numerical Algebra, Control and Optimization }(NACO), 1, 167--185.

\bibitem[Le Thi et al. c (2012)]{Lethietal2012c} Le Thi H.A., Moeini, M.,
Pham Dinh, T. \& Joaquim, J. (2012). A DC programming approach for solving
the symmetric eigenvalue complementarity problem. \textit{Computational
Optimization and Applications,} 51:3, 1097--1117

\bibitem[Le Thi and Moeini (2012)]{Long-Short PortfolioJOTA12} Le Thi, H.A.
\& Moeini, M. (2012). Long-Short Portfolio Optimization Under Cardinality
Constraints by Difference of Convex Functions Algorithm. \textit{Journal of
Optimization Theory \& Applications}, DOI 10.1007/s10957-012-0197-0 October
2012, 27 pages.

\bibitem[Le Thi et al. a (2013)]{LeThietal2013} Le Thi, H.A., Le, H.M., Pham
Dinh, T. \& Huynh, V.N. (2013). Binary classification via spherical
separator by DC programming and DCA. \textit{Journal of Global Optimization}%
, 56:4, 1393--1407.

\bibitem[Le Thi et al. b (2013)]{LTNL13} Le Thi, H.A., Nguyen, T.B.T \& Le
H.M. (2013). Sparse signal recovery by Difference of Convex functions
Algorithms. \textit{Lecture Notes in Computer Science}, ISBN
978-3-642-36542-3, 387--397.

\bibitem[Le Thi et al. c (2013)]{LTLP13} Le Thi, H.A., Le, H.M. \& Pham Dinh, T.
(2014). Feature Selection in machine learning: an exact penalty approach
using a Difference of Convex function Algorithm, submitted.

\bibitem[Le Thi and Pham Dinh (2013)]{LTP2013} Le Thi, H.A., Pham
Dinh, T. (2013). DC programming approaches for Distance Geometry problems.
in ''Distance Geometry: Theory, Methods and Applications'', Mucherino, A;
Lavor, C; Liberti, L.; Maculan, N. (Eds), Springer, 225--290.

\bibitem[Le Thi et Nguyen (2013)]{LTN13} Le Thi, H.A. \& Nguyen M.C. (2013).
Efficient algorithms for Feature Selection in Multi-class Support Vector
Machine. \textit{Advanced Computational Methods for Knowledge Engineering},
Studies in Computational Intelligence 479, Springer.

\bibitem[Le Thi et al. (2014)]{LeThietal2014} Le Thi, H.A., Le, H.M., Pham
Dinh, T. \& Lauer, F. (2014). A DC programming algorithm for switched linear
regression. \textit{IEEE Transactions on Automatic Control}, 99, forthcoming.

\bibitem[Le Thi et al. (2014)]{ANT14} Le Thi, H.A., Huynh, V.N. \& Pham
Dinh, T. (2014). DC Programming and DCA for solving general DC Programs.
Proceedings of 2nd International Conference on Computer Science, Applied
Mathematics and Applications(ICCSAMA 2014), in Press, (35 pages).


\bibitem[Liu et al. b(2005)]{Liu:pr05} Liu, Y. \& Zheng, Y.F. (2005). FS
SFS: A novel feature selection method for support vector machines. \textit{%
Pattern Recognition}, doi:10.1016/j.patcog.2005.10.006.

\bibitem[Liu et al. (2006)]{Liu:jasa06} Liu, Y. \& Shen, X. (2006).
Multicategory$\psi $-Learning. \textit{Journal of the American Statistical
Association}, 101, 500--509.

\bibitem[Mallat and Zang (1993)]{Mallat-Zang93} Mallat, S. \& Zhang, Z.
(1993). Matching pursuit in a time-frequency dictionary. \textit{IEEE Trans.
Signal Processing}, 41(12), 3397--3415.

\bibitem[Mangasarian (1996)]{Mangasarian96} Mangasarian, O.L. (1996).
Machine learning via polyhedral concave minimization, in ''Applied
Mathematics and Parallel Computing -- Festschrift for Klaus Ritter'', H.
Fischer, B. Riedmueller, S. Schaeffler, editors, Physica-Verlag, Germany,
175---188.

\bibitem[Mohri and Median (2014)]{Mohri14} Mohri, M.,\& Medina, A.M. (2014).
Learning Theory and Algorithms for Revenue Optimization in Second-Price
Auctions with Reserve. \textit{Proceeding ICML'14 Proceedings of the 31th
international conference on Machine learning},
http://arxiv.org/abs/1310.5665.

\bibitem[Natarajan (1995)]{Natarajan95} Natarajan, B.K. (1995). Sparse
approximate solutions to linear systems. \textit{SIAM J. Comp.}, 24,
227--234.

\bibitem[Newmann et al. (2005)]{Neumann05} Neumann, J., Schn\"{o}rr G.,
Steidl, G. (2005). Combined SVM-based feature selection and classification. 
\textit{Machine Learning}, 61(1-3), 129--150.

\bibitem[Niu et al. d (2012)]{Niuetal2012} Niu, Y.S., Pham Dinh, T., Le Thi
H.A. \& Judice, J. (2012). Efficient DC Programming Approaches for
Asymmetric Eigenvalue Complementarity Problem, Optimization Methods and
Software, DOI:10.1080/10556788.2011.645543, Online first Feabruary 2012.

\bibitem[Ong and Le Thi (2012)]{Ong12} Ong, C.S \& Le Thi H.A. (2013).
Learning sparse classifiers with Difference of Convex functions Algorithms. 
\textit{Optimization Methods and Software}, 28:4, 830--854.

\bibitem[Pati et al. (1993)]{Pati93} Pati, Y.C., Rezaifar, R. \&
Krishnaprasa, P.S. (1993). Orthogonal matching pursuit: recursive function
approximation with applications to wavelet decomposition. \textit{27th
Asilomar Conf. on Signals, Systems and Comput.}, Nov. 1993.

\bibitem[Pham Dinh and Le Thi (1997)]{Acta} Pham Dinh, T. \& Le Thi, H.A
(1997). Convex analysis approach to d.c. programming: Theory, Algorithm and
Applications. \textit{Acta Mathematica Vietnamica}, 22, 289--355.

\bibitem[Pham Dinh and Le Thi (1998)]{PLT98} Pham Dinh, T. \& Le Thi, H.A.
(1998). DC optimization algorithms for solving the trust region subproblem. 
\textit{SIAM J.Optimization}, 8, 476--505.

\bibitem[Pham Dinh et al. (2010)]{BQPJOGO10} Pham Dinh, T., Nguyen Canh, N.
\& Le Thi, H.A. (2010). An efficient combination of DCA and B\&B using
DC/SDP relaxation for globally solving binary quadratic programs. \textit{%
Journal of Global Optimization}, 48:4, 595--632.

\bibitem[Pham Dinh and Le Thi (2014)]{PLT13} Pham Dinh, T. \&Le Thi, H.A.
(2014). Recent Advances in DC Programming and DCA. \textit{Transactions on
Computational Collective Intelligence}, 8342, 1-37.



\bibitem[Rao and Kreutz-Delgado (1999)]{Rao99} Rao, B.D. \& Kreutz-Delgado,
K. (1999). An affine scaling methodology for best basis selection. \textit{%
IEEE Trans. Signal Processing}, 47, 87--200.

\bibitem[Rao et al. (2003)]{Rao03} Rao, B.D., Engan, K., Cotter, S.F.,
Palmer. J., \& KreutzDelgado, K. (2003). Subset selection in noise based on
diversity measure minimization. \textit{IEEE Trans. Signal Processing},
51(3), 760--770.


\bibitem[Rinaldi et al. (2010)]{Rinaldi10} Rinaldi, F., Schoen, F. \&
Sciandrone, M. (2010). Concave programming for minimizing the zero-norm over
polyhedral sets. \textit{Comput. Optim. Appl.}, 46(3), 467--486.

\bibitem[Rockafellar (1970)]{Rockafellar} Rockafellar, R.T (1970). Convex
Analysis. Princeton: Princeton University.



\bibitem[Schmidt et al. (2007)]{schmidt} Schmidt, M., Fung, G. \& Rosales,
G. (2007). Fast Optimization Methods for L1 Regularization: A Comparative
Study and Two New Approaches. \textit{Proceedings of Machine Learning: ECML
2007}, Lecture Notes in Computer Science, 4701, 286--297.

\bibitem[Sriperumbudur et al. (2007)]{Sriperumbudur} Sriperumbudur, B.K.,
Torres, D.A. \& Lanckriet, R.G. (2007). Sparse eigen methods by D.C.
programming. \textit{Proceeding ICML '07, Proceedings of the 24th
international conference on Machine learning}, 831-838.

\bibitem[Takhar et al. (2006)]{Takhar06} Takhar, D., Laska, J.N., Wakin,
M.B., Duarte, M.F., Baron, D., Sarvotham, S., Kelly, K.F. \& Baraniuk, R.G
(2006). A New Compressive Imaging Camera Architecture using Optical-Domain
Compression. \textit{Computational Imaging IV at IS\&T/SPIE Electronic
Imaging}, San Jose, California, January 2006.

\bibitem[Tan et al. (2010)]{Tan10} Tan, M., Wang, L. \& Tsang, I.W. (2010).
Learning sparse svm for feature selection on very high dimensional datasets. 
\textit{ICML 2010}.

\bibitem[Thiao et al (2008)]{Thiao-et-al08} Thiao, M., Pham Dinh, T. \& Le
Thi, H.A. (2008). DC Programming approach for solving a class of Nonconvex
Programs dealing with zero-norm. \textit{Modelling, Computation and
Optimization in Information Systems ans Management Science,} CCIS 14,
348--357, Springer-Verlag.

\bibitem[Thiao et al (2011)]{Thiao-et-al11} Thiao M., Pham Dinh T., \& Le
Thi H.A. (2010). A DC programming approach for Sparse Eigenvalue Problem. 
\textit{Proceedings} \textit{of the 27th} \textit{International Conference
on Machine Learning, ICML 2010}, 1063--1070.

\bibitem[Tibshirani (1996)]{TIB96} Tibshirani, R. (1996). Regression
shrinkage and selection via the lasso. \textit{J. Roy. Stat. Soc.}, 46,
431--439.




\bibitem[Weston et al. (2003)]{Weston03} Weston, J., Elisseeff, Scholkopf,
A.B. \& Tipping, M. (2003). Use of the Zero-Norm with Linear Models and
Kernel Methods. \textit{Journal of Machine Learning Research}, 3, 1439--1461.

\bibitem[Zhang (2009)]{Zhang09} Zhang, T. (2009). Some sharp performance
bounds for least squares regression with regularization. \textit{Ann.
Statist.}, 37, 2109--2144.

\bibitem[Zhang et al. (2006)]{Zhang06} Zhang, H.H, Ahn, J., Lin, X. \& Park,
C. (2006). Gene selection using support vector machines with non-convex
penalty. \textit{Bioinformatics}, 2(1), 88--95.

\bibitem[Zhu et al. (2004)]{Zhu04} Zhu, J., Rosset, S., Hastie, T. \&
Tibshirani, R. (2004). 1-norm support vector machines. In S. Thrun, L.Saul,
\& B. Scholkopf (Eds.), \textit{Adv. neur. inf. proc. sys.}, 16 Cambridge,
MA: MIT Press.

\bibitem[Zou (2006)]{ZOU06} Zou, H. (2006). The adaptive lasso and its
oracle properties. \textit{J. Amer. Stat. Ass.}, 101, 1418--1429.

\bibitem[Zou and Li (2008)]{Zou08} Zou, H. \& Li, R. (2008). One-step sparse
estimates in nonconcave penalized likelihood models. \textit{Ann. Statist.},
36(4), 1509--1533.
\end{thebibliography}
\end{document}